\documentclass[12pt]{amsart}
\usepackage{a4wide,enumerate,graphicx}
\usepackage{color}
\allowdisplaybreaks

\let\pa\partial  
\let\na\nabla  
\let\eps\varepsilon  
\newcommand{\N}{{\mathbb N}}  
\newcommand{\R}{{\mathbb R}}

\newcommand{\T}{{\mathbb T}}
\newcommand{\ent}{{\mathcal H}}

\newtheorem{theorem}{Theorem}   
\newtheorem{lemma}[theorem]{Lemma}   
\newtheorem{proposition}[theorem]{Proposition}   
\newtheorem{remark}[theorem]{Remark}

\begin{document}  

\title[A discrete Bakry-Emery method]{A discrete Bakry-Emery
method and its application to the porous-medium equation}

\author{Ansgar J\"ungel}
\address{Institute for Analysis and Scientific Computing, Vienna University of  
	Technology, Wiedner Hauptstra\ss e 8--10, 1040 Wien, Austria}
\email{juengel@tuwien.ac.at} 
\author{Stefan Schuchnigg}
\address{Institute for Analysis and Scientific Computing, Vienna University of  
	Technology, Wiedner Hauptstra\ss e 8--10, 1040 Wien, Austria}
\email{stefan.schuchnigg@tuwien.ac.at} 

\date{\today}

\thanks{The authors acknowledge partial support from   
the Austrian Science Fund (FWF), grants P22108, P24304, and W1245, and    
the Austrian-French Program of the Austrian Exchange Service (\"OAD)} 

\begin{abstract}
The exponential decay of the relative entropy associated to a fully discrete 
porous-medium equation in one space dimension
is shown by means of a discrete Bakry-Emery approach. 
The first ingredient of the proof is an abstract discrete Bakry-Emery method, 
which states conditions on a sequence under which the exponential decay of the 
discrete entropy follows. The second ingredient is a new nonlinear
summation-by-parts formula which is inspired by systematic integration by parts
developed by Matthes and the first author. Numerical simulations illustrate the
exponential decay of the entropy for various time and space step sizes.
\end{abstract}

% \paragraph{Keywords:}  
 
% \paragraph{AMS classification:}  

\maketitle

%%%%%%%%%%%%%%%%%%%%%%%%%%%%%%%%%%%%%%%%%%%%%%%%%%%%%%%%%%%%%%%%%%%%%%%%%%%%%%%

\section{Introduction}

The Bakry-Emery method allows one to establish convex Sobolev
inequalities and to compute exponential decay rates towards equilibrium 
for solutions to diffusion equations \cite{BaEm85,BGL14}. 
The key idea of Bakry and Emery
is to differentiate a so-called entropy functional twice with respect to time 
and to relate the second-order derivative to the entropy production. 
Our aim is to develop a discrete version of this
technique, and in this paper, we present a step forward in this direction.

The study of discrete Bakry-Emery methods and related topics is rather recent. 
Caputo et al.\ \cite{CDP09} computed
exponential decay rates for time-continuous Markov processes, using the
Bochner-Bakry-Emery method. Given a stochastic process with density $u(t)$
and the entropy functional $H_c(u(t))$, the core of the Bakry-Emery approach is to 
find a constant $\lambda>0$ such that the inequality $d^2H_c/dt^2\ge -\lambda dH_c/dt$
holds for all time. Integrating this inequality, one may show that
$dH_c/dt\le -\lambda H_c$ which implies that $H_c(u(t))\le e^{-\lambda t}H_c(u(0))$ 
for all $t>0$, i.e., the entropy decays exponentially fast along $u(t)$.
The relation between $d^2H_c/dt^2$ and $dH_c/dt$ is achieved in \cite{CDP09}
by employing a discrete Bochner-type identity
which replaces the Bochner identity of the continuous case.
The Bochner-Bakry-Emery method was extended by Fathi and Maas in 
\cite{FaMa16} in the context of Ricci curvature bounds and used by the authors of
\cite{JuYu16} to derive discrete Beckner inequalities. 

Another approach has been suggested by Mielke \cite{Mie13}. He investigated
geodesic convexity properties of nonlocal transportation distances on probability
spaces such that continuous-time Markov chains can be formulated as gradient flows.
Very related results have been obtained independently by Chow et al.\ \cite{CHLZ12} 
and Maas \cite{Maa11}. 
The geodesic convexity property implies exponential decay rates \cite{AGS05}. 
Mielke showed that the inequality 
$d^2H_c/dt^2\ge -\lambda dH_c/dt$ is equivalent to the positive semi-definiteness of 
a certain matrix such that matrix algebra can be applied. This idea was extended
recently to certain nonlinear Fokker-Planck equations \cite{CJSa16}. 

All these examples involve spatial semi-discretizations of diffusion
equations. Temporal semi-discretizations often employ the implicit Euler scheme
since it gives entropy dissipation, $dH_c/dt\le 0$, under rather general conditions;
see, e.g., the implicit Euler finite-volume approximations in 
\cite{CJS16,Gli08}. Entropy-dissipating higher-order 
semi-discretizations have been analyzed in \cite{Emm09,JuMi15,JuSc16}. 
However, there seem to be no results for fully discrete schemes using the
Bakry-Emery approach. In this paper, we make a first step to fill this gap.

In order to understand the mathematical difficulty in fully discrete schemes,
consider the abstract Cauchy problem
\begin{equation}\label{1.eq}
  \pa_t u + A(u)=0, \quad t>0, \quad u(0)=u^0,
\end{equation}
where $A:D(A)\to X'$ is a (nonlinear) operator defined on its domain 
$D(A)\subset X$ of the Banach space $X$ with dual $X'$. If the dual
product $\langle A(u),H_c'(u)\rangle$ is nonnegative, where 
$H_c'(u)$ is the (Fr\'echet) derivative
of the entropy and $u(t)$ a solution to \eqref{1.eq}, then
$$
  \frac{dH_c}{dt} = \langle\pa_t u,H_c'(u)\rangle 
	= -\langle A(u),H_c'(u)\rangle \le 0,
$$
showing entropy dissipation.
Next, consider the implicit Euler scheme
$$
  \tau^{-1}(u^k - u^{k-1}) + A_h(u^k) = 0, \quad k\in\N,\ \tau>0,
$$
where $u^k$ is an approximation of $u(k\tau)$ and $A_h$ is an approximation
of $A$ still satisfying $\langle A_h(u^k),H'(u^k)\rangle\ge 0$. Here, $H(u^k)$
is the discrete entropy, which is supposed to be convex. Then
entropy dissipation is preserved by the scheme since
\begin{equation}\label{1.ed}
  H(u^k)-H(u^{k-1}) \le \langle u^k-u^{k-1},H'(u^k)\rangle
	= -\tau\langle A_h(u^k),H'(u^k)\rangle \le 0.
\end{equation}
The problem is to estimate
the discrete analog of $d^2H_c/dt^2$. It turns out that the inequality in
\eqref{1.ed} is too weak, we need an equation; see Section \ref{sec.cBE} for
details. We overcome this difficulty by developing two ideas.

The {\em first idea} is to identify the elements which are necessary to build an
abstract discrete Bakry-Emery method. Unlike in the continuous case,
we distinguish between the discrete entropy production 
$P:=-\tau^{-1}(H(u^k)-H(u^{k-1}))$ and the Fisher information
$F:=\langle A_h(u^k),H'(u^k)\rangle$. We explain this difference
in Section \ref{sec.abstract}.

The Bakry-Emery method
relies on an estimate of $\tau^{-1}(F(u^k)-F(u^{k-1}))$, which approximates
$d^2H_c/dt^2$. For this estimate, discrete versions of suitable integrations
by parts and chain rules are necessary. Our {\em second idea} is to ``translate''
a nonlinear integration-by-parts formula to the discrete case, using the
systematic integration by parts method of \cite{JuMa06}. This leads to a new
inequality for numerical three-point schemes as explained next.

Again, consider first the continuous case. We show in Lemma \ref{lem.ineq} that
for all $(A,B)\in R_c:=\{(A,B)\in\R^2:(2A-B-1)(A+B-2)<0\}$ and 
all smooth positive functions $w$,
$$
  \int_\T w_{xx} (w^A)_{xx} w^B dx \ge \kappa_c\int_\T w^{A+B-1}w_{xx}^2 dx,
$$
where the constant $\kappa_c>0$ depends on $A$ and $B$; see \eqref{kappac} below.
The proof is based on systematic integration by parts \cite{JuMa06}. 
The discrete counterpart is the following inequality: For any $0<\eps\le 1$, 
there exists a region $R$ of admissible values $(A,B)$, containing the line $A=1$,
such that for all $w_0,\ldots,w_{N+1}\ge 0$ with $w_N=w_0$, $w_{N+1}=w_1$, 
\begin{align}
  \sum_{i=1}^{N} & (w_{i+1}-2w_i+w_{i-1})(w_{i+1}^A-2w_i^A+w_{i-1}^A)w_i^B 
	\nonumber \\
	&\ge \kappa\sum_{i=1}^{N} \min_{j=i,i\pm 1}w_j^{A+B-1}(w_{i+1}-2w_i+w_{i-1})^2,
	\label{1.ineqd}
\end{align}
where $\kappa=\eps A$; see Lemma \ref{lem.ineqd}. Interestingly, the inequality
is not true for each term but only with the sum.
The admissible set $R$ for \eqref{1.ineqd}
is generally smaller than $R_c$; see Section \ref{sec.ibp}. 
We conjecture that $R=R_c$ for $\kappa=0$.

Inequality \eqref{1.ineqd} is the first nonlinear summation-by-parts
formula derived from a systematic method. We believe that this idea will lead
to a whole family of new finite-difference inequalities useful in numerical analysis,
and we will explore this in a future work.

We apply the abstract discrete Bakry-Emery method in Section \ref{sec.pme}
to an implicit Euler finite-difference approximation of the porous-medium equation
$$
  \pa_t u = (u^\beta)_{xx} \quad\mbox{in }\T,\ t>0, \quad u(0)=u^0\ge 0,
$$
where $\beta>1$ and $\T$ is the one-dimensional torus. We assume, for simplicity,
that $\mbox{meas}(\T)=1$ and identify $\T$ with $[0,1]$. 
The entropy functional is $H_c(u)=\int_\T(u^\alpha-\overline{u}^\alpha)dx/(\alpha-1)$, 
where $\alpha>0$ and $\overline{u}=\int_\T u^0 dx$
is the constant steady state. We show in Proposition \ref{prop.ent} that
$H_c(u(t))$ decays exponentially fast to zero for all $(\alpha,\beta)\in S_c$, where
\begin{equation}\label{1.Sc}
  S_c = \{(\alpha,\beta)\in\R_+^2: \alpha+\beta>1,\ -2<\alpha-\beta<1\},
\end{equation}
with a decay rate depending on $(\alpha,\beta)$ and $\min_\T u^{\beta-1}$.

To overcome the difficulty with the entropy production inequality, 
we introduce the new variable $v=u^\alpha$ and write
the porous-medium equation in the form
\begin{equation}\label{1.eqv}
  \pa_t v = \alpha u^{\alpha-1}\pa_t u 
	= \alpha v^{(\alpha-1)/\alpha}(v^{\beta/\alpha})_{xx}.
\end{equation}
The advantage of this formulation is that the entropy becomes {\em linear}
in the variable $v$. thus avoiding inequality \eqref{1.ed}.

We discretize \eqref{1.eqv} by an implicit Euler finite-difference scheme.
Let $\tau>0$ be the time step, $h>0$ the space step, and let $v_i^k=(u_i^k)^\alpha$ 
be an approximation of $(h^{-1}\int_{(i-1)h}^{ih}u(x,k\tau) dx)^\alpha$, 
$i=1,\ldots,N$. The iterative scheme reads as
\begin{equation}\label{1.scheme}
  v_i^k - v_i^{k-1} = \tau h^{-2}\alpha (v_i^k)^{(\alpha-1)/\alpha}
	\big((v_{i+1}^k)^{\beta/\alpha} - 2(v_i^k)^{\beta/\alpha} 
	+ (v_{i-1}^k)^{\beta/\alpha}\big),
\end{equation}
where $i=1,\ldots,N$, $k\in\N$, and $v_N^k=v_0^k$, $v_{N+1}^k=v_1^k$.
We show in Lemma \ref{lem.ex} the existence of solutions to \eqref{1.scheme}
as well as the preservation of nonnegativity. However, the total mass
$h\sum_{i=1}^N u_i^k = h\sum_{i=1}^N (v_i^k)^{1/\alpha}$ is not conserved,
which is the price that we have to pay for the estimation of the entropy production.
We discuss this point in Section \ref{sec.num}. Our main result reads as follows.

\begin{theorem}\label{thm.main}
Let $v^k=(v_i^k)$ be a 
nonnegative solution to \eqref{1.scheme} and set $u_i^k=(v_i^k)^{1/\alpha}$.
Let $0<\eps<1$. Then there exist a region $S\subset(0,\infty)^2$, 
containing the line $\alpha-\beta=1$, and a number $U>0$
such that all $(\alpha,\beta)\in S$ with $\alpha>1$ and $\beta\ge 1$, it holds that
$$
  \ent(u^k) \le \ent(u^0)e^{-\eta\lambda k\tau}, \quad k\in\N,
$$
where
$$
  \ent(u^k)=\frac{h}{\alpha-1}\sum_{i=1}^N \big((u_i^k)^\alpha-U^\alpha\big)dx
$$
is the discrete (relative) entropy,
$$
  \eta = \frac{\log(1+k\tau)}{k\tau}, \quad
	\lambda = \frac{8\eps(\alpha-1)\beta^2}{C_p(\alpha+\beta-1)^2}
	\min_{i=1,\ldots,N}u_i^{\beta-1}, 
$$
and $C_p=h^2/(4\sin^2(h\pi))\ge 1/(4\pi^2)$ the the discrete Poincar\'e
constant. Moreover, the total mass
$h\sum_{i=1}^N u_i^k$ is increasing in $k$ and converges to $U$ as $k\to\infty$.
\end{theorem}

\begin{remark}[Exponential versus algebraic decay]\rm
The exponential decay rate depends on the minimum of the solution, 
which is not surprising. Indeed, because of the degeneracy, we cannot generally
expect exponential decay; an example is the Barenblatt solution. 
Algebraic decay rates for implicit Euler finite-volume schemes have been
derived in, e.g., \cite{CJS16}.
When the minimum is positive, the equation is no longer degenerate, 
and exponential decay follows.
\qed
\end{remark}

\begin{remark}[Shannon entropy]\rm
Unfortunately, the theorem does not apply to the Shannon entropy 
$h\sum_i u_i\log u_i$, corresponding to $\alpha\to 1$, since 
$\lambda\to 0$ as $\alpha\to 1$. The reason is that for $\alpha\to 1$,
the entropy production $P$
cannot be bounded from above by the Fisher information $F$ and so, Assumption
A1 of our abstract Bakry-Emery method does not hold; see Section \ref{sec.dBE}.
\qed
\end{remark}

\begin{remark}[Discrete gradient flow]\rm
Erbar and Maas \cite{ErMa14} showed that 
the gradient flow of the Shannon entropy with respect to a nonlocal
transportation measure equals the discrete porous-medium equation in one
space dimension. The porous-medium equation in several space dimensions
was solved by Benamou et al.~ \cite{BCMO16} by providing a spatial discretization 
of this equation as a convex optimization problem. In both references, no
decay rates have been derived.
\qed
\end{remark}

The set $S$ is illustrated in Figure \ref{fig.alphabeta} for two different
values of $\eps$. Numerical computations indicate that 
$S$ converges (in the set theoretical sense) to the set $S_c$ defined in \eqref{1.Sc} 
if $\eps\to 0$ but for fixed $\eps>0$, $S$ is strictly contained in $S_c$. 

\begin{figure}[ht]
\includegraphics[width=80mm]{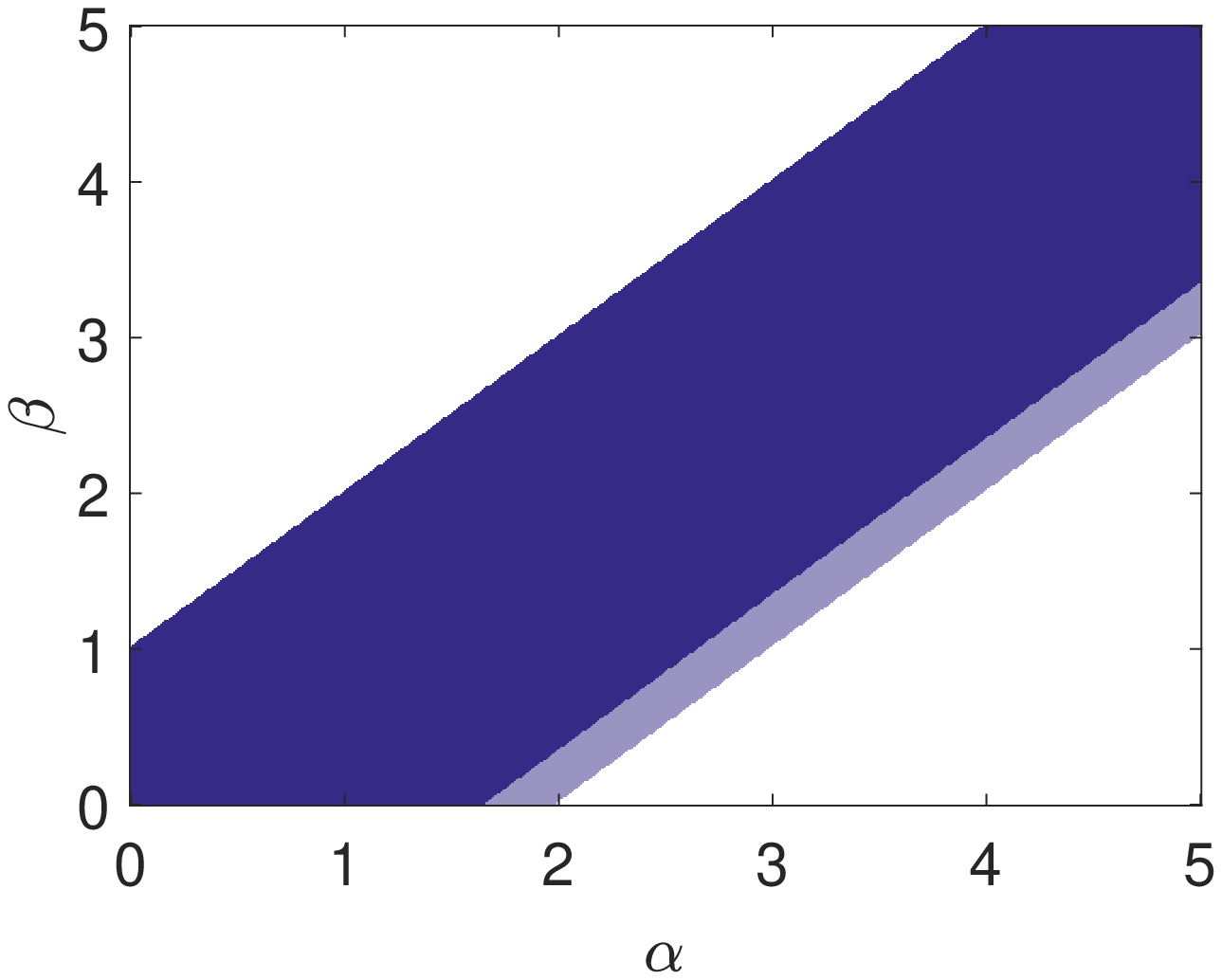}
\includegraphics[width=80mm]{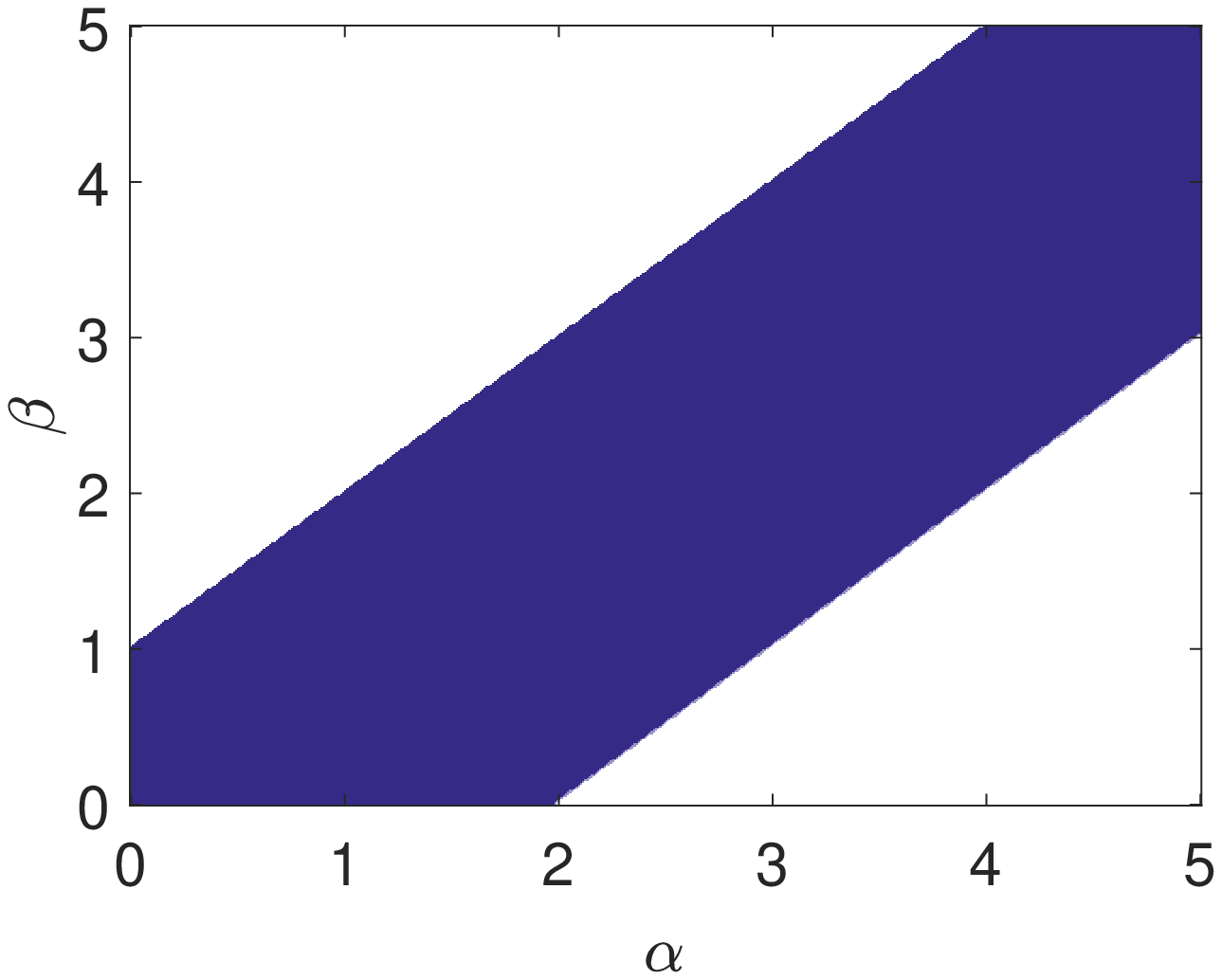}
\caption{Admissible region $S$ for $\eps=1/4$ (left) and $\eps=1/100$ (right).
The set $S_c$, defined by $-1<\alpha-\beta<2$,
is shown in light blue for comparison; it contains the dark blue region $S$.}
\label{fig.alphabeta}
\end{figure}

The paper is organized as follows. The abstract Bakry-Emery result is presented in
Section \ref{sec.abstract}, and in Section \ref{sec.ibp}, inequality
\eqref{1.ineqd} is verified. Theorem \ref{thm.main} is proved
in Section \ref{sec.pme}. Numerical examples are presented in
Section \ref{sec.num}, and some auxiliary inequalities are recalled in
the Appendix.

%%%%%%%%%%%%%%%%%%%%%%%%%%%%%%%%%%%%%%%%%%%%%%%%%%%%%%%%%%%%%%%%%%%%%%%%%%%%%%%

\section{An abstract Bakry-Emery method}\label{sec.abstract}

In this section, we present our abstract result. In order to identify the
key ingredients of the Bakry-Emery method, we recall the basic ideas for
continuous evolution equations.

\subsection{The continuous Bakry-Emery method}\label{sec.cBE}

Let us first consider the abstract Cauchy problem
\begin{equation}\label{2.eq}
  \pa_t u + A(u) = 0, \quad t>0, \quad u(0)=u^0.
\end{equation}
The nonlinear operator $A$ is defined on some domain $D(A)$ of a Banach space $X$.
We do not specify the properties of $A$ nor its domain since they are not needed
in the following. As mentioned in the introduction, 
the idea of the Bakry-Emery method is to differentiate the entropy
functional $H_c:D(A)\to[0,\infty)$ twice with respect to time along solutions
to \eqref{2.eq}. 
We define the entropy production $P_c(u(t)):=-\frac{d}{dt}H_c(u(t))$.
If $\langle A(u),H_c'(u)\rangle\ge 0$ holds for all $u\in D(A)$
($\langle\cdot,\cdot\rangle$ is the dual product in $X$) then
$$
  P_c(u) = -\langle\pa_t u,H_c'(u)\rangle = \langle A(u),H_c'(u)\rangle \ge 0,
$$
i.e., the entropy production is nonnegative and the entropy is nonincreasing
along solutions to \eqref{2.eq}. We call $F_c(u):=\langle A(u),H_c'(u)\rangle$ 
the generalized Fisher information since if $A(u)=-\Delta u$ on $\T^d$ and 
$H_c(u)=\int_{\T^d} u(\log u-1)dx$, we obtain the Fisher information
functional $F_c(u)=4\int_{\T^d}|\na\sqrt{u}|^2 dx$.
Clearly, $P_c(u(t))=F_c(u(t))$ along solutions $u(t)$ to \eqref{2.eq}.

Differentiating $F_c$ gives
\begin{align*}
  \frac{dF_c}{dt} 
	&= \langle A'(u)[\pa_t u],H_c'(u)\rangle + \langle A(u),H_c''(u)\pa_t u\rangle \\
	&= -\langle A'(u)[A(u)],H_c'(u)\rangle - \langle A(u),H_c''(u)A(u)\rangle,
\end{align*}
where $A'(u)$ is the Fr\'echet derivative of $A$ at $u$. If the functional
inequality
\begin{equation}\label{2.fi}
  \langle A'(u)[A(u)],H_c'(u)\rangle + \langle A(u),H_c''(u)A(u)\rangle
	\ge \lambda_c\langle A(u),H_c'(u)\rangle
\end{equation}
holds for some $\lambda_c>0$ then
\begin{equation}\label{2.fish}
  \frac{dF_c}{dt} \le -\lambda_c\langle A(u),H_c'(u)\rangle = -\lambda_c F_c,
\end{equation}
and we conclude exponential decay of $t\mapsto F_c(u(t))$ with rate $\lambda_c>0$.
In particular, $\lim_{t\to\infty}F_c(u(t))=0$. Then, integrating the previous
inequality over $(t,\infty)$, it follows that
$$
  \frac{dH_c}{dt}(u(t)) = -F_c(u(t))\le -\lambda_c\int_t^\infty F_c(u(s))ds 
	= \lambda_c\int_t^\infty \frac{dH_c}{dt}(u(s))ds.
$$
Assuming that also
\begin{equation}\label{2.H}
  \lim_{t\to\infty}H_c(u(t))=0, 
\end{equation}
we conclude that
$$
  \frac{dH_c}{dt}(u(t)) \le -\lambda_c H_c(u(t)), \quad t>0,
$$
and by Gronwall's lemma, $t\mapsto H_c(u(t))$ converges exponentially fast to zero
with rate $\lambda_c$.

We see that two assumptions are essential: the functional inequality
\eqref{2.fi} and the limit \eqref{2.H}. On the discrete level, we need
to distinguish between the (discrete) entropy production and the (discrete)
Fisher information since $dH_c/dt$ and $\langle A(u),H_c'(u)\rangle$
may differ on the discrete level. 
We assume that both functionals can be estimated
by each other. Instead of the functional inequality \eqref{2.fi} we assume
a discrete version of inequality \eqref{2.fish}. Finally, a discrete version
of \eqref{2.H} is required. 

\subsection{A discrete Bakry-Emery method}\label{sec.dBE}

Let two functions $H:\R^{N}\to[0,\infty)$ and
$F:\R^{N}\to[0,\infty)$ be given and define
$P(v):=P(v;w)=-\tau^{-1}(H(v)-H(w))$, where $v$, $w\in\R^{N}$ and $\tau>0$.
We call $H$ an entropy, $F$ the Fisher information, and $P$ the entropy production.
The following result does not need any reference to the solution of a discrete
problem.

\begin{proposition}\label{prop.BE}
Let $(v^k)\subset\R^N$ be any sequence. We assume that
\begin{description}
\item[\rm A1] There exist $C_m$, $C_M>0$ such that 
$C_m F(v^k)\le P(v^k)\le C_M F(v^k)$ for all $k\in\N$.
\item[\rm A2] 
There exists $\kappa>0$ such that $F(v^k)-F(v^{k-1})\le -\tau\kappa F(v^k)$
for all $k\in\N$.
\item[\rm A3] $\lim_{k\to\infty}H(v^k)=0$.
\end{description}
Then 
$$
  H(v^k) \le e^{-\eta\lambda k\tau} H(v^0), \quad k\in\N,
$$
where $\lambda = (C_m/C_M)\kappa$ and $\eta = \log(1+\tau\lambda)/(\tau\lambda)\in(0,1)$.
\end{proposition}

The discrete decay rate $\lambda$ is generally smaller than the 
decay rate $\kappa$ of the Fisher information,
since $\eta<1$ and we may have $C_m<C_M$. If the entropy production
and the Fisher information coincide, i.e.\ $C_m=C_M=1$, then $\lambda=\kappa$.

\begin{proof}
By Assumption A2, it follows that $\lim_{k\to\infty}F(v^k)=0$.
By Assumption A2 again and the second inequality in Assumption A1, we have
$$
  F(v^k)-F(v^{k-1}) \le -\tau\kappa F(v^k) \le -\tau\kappa C_M^{-1}P(v^k)
	= \kappa C_M^{-1}\big(H(v^k)-H(v^{k-1})\big).
$$
Taking the sum from $k=\ell+1$ to $k=m>\ell+1$, we find that
$$
  F(v^m)-F(v^\ell) \le \kappa C_M^{-1}(H(v^m)-H(v^\ell)).
$$
Passing to the limit $m\to\infty$, observing that $\lim_{m\to\infty}F(v^m)=0$
and, by Assumption A3, $\lim_{m\to\infty}H(v^m)=0$, we deduce that
$$
  F(v^\ell) \ge \kappa C_M^{-1} H(v^\ell),
$$
which holds for all $\ell\in\N$. It remains to use the first inequality in
Assumption A1 to conclude that
$$
  H(v^k)-H(v^{k-1}) = -\tau P(v^k)
	\le -\tau C_m F(v^k) \le -\tau C_m \kappa C_M^{-1} H(v^k) = -\tau\lambda H(v^k).
$$
We deduce that $H(v^k)\le (1+\lambda\tau)^{-k}H(v^0)
=\exp(-\eta\lambda k\tau)H(v^0)$, finishing the proof.
\end{proof}

%%%%%%%%%%%%%%%%%%%%%%%%%%%%%%%%%%%%%%%%%%%%%%%%%%%%%%%%%%%%%%%%%%%%%%%%%%%%%%%

\section{A nonlinear summation-by-parts formula}\label{sec.ibp}

To apply the abstract Bakry-Emery method to the porous-medium equation,
we need to verify the assumptions of Proposition \ref{prop.BE}. The key
condition is Assumption A2. To verify it, we ``translate'' some
integrations by parts to the discrete level. It is convenient to investigate the
continuous situation first in order to formulate the discrete
formula that is needed to show Assumption A2.

Consider the nonlinear diffusion equation
\begin{equation}\label{3.pme}
  \pa_t u = (u^\beta)_{xx},\quad t>0, \quad u(0)=u^0\ge 0\quad\mbox{in }\T,
\end{equation}
where $\beta>0$, and introduce the (relative) entropy
$$
  H_{c}(u)=\frac{1}{\alpha-1}\int_\T (u^\alpha-\overline{u}^\alpha) dx, \quad
	\alpha>0.
$$
Here, $\overline{u}=\int_\T u^0(x)dx$ is the constant steady state.
(Recall that $\mbox{meas}(\T)=1$.) 

\begin{proposition}\label{prop.ent}
Let $\beta\neq 1$, $\alpha+\beta-1>0$, and $-1<\alpha-\beta<2$. Then, 
for any positive smooth solution to \eqref{3.pme},
$$
  H_c(u(t))\le H_c(u^0)e^{-\lambda_c t}, \quad t>0,
$$
where
$$
	\lambda_c = \frac{16\pi^2\alpha\beta\kappa_c}{\alpha+\beta-1}
	\min_\T u^{\beta-1} \ge 0,
	\quad \kappa_c = -\frac{4\beta(\alpha-\beta-2)}{(\alpha+\beta-1)(\alpha-\beta+1)} >0.
$$
\end{proposition}

\begin{proof}
Integrating by parts, the time derivatives of $H_c(u(t))$ become
\begin{align*}
  \frac{dH_c}{dt}
	&= \frac{\alpha}{\alpha-1}\int_\T u^{\alpha-1}(u^\beta)_{xx} dx
	%= -\alpha\beta\int_\T u^{\alpha+\beta-3}u_x^2 dx \\
	= -\frac{4\alpha\beta}{(\alpha+\beta-1)^2}
	\int_\T \big(u^{(\alpha+\beta-1)/2}\big)_x^2 dx, \\
	\frac{d^2H_c}{dt^2}
	&= -\frac{8\alpha\beta}{(\alpha+\beta-1)^2}
	\int_\T \big(u^{(\alpha+\beta-1)/2}\big)_x
	\bigg(\frac{\alpha+\beta-1}{2}u^{(\alpha+\beta-3)/2}\pa_t u\bigg)_x dx \\
	&= \frac{4\alpha\beta}{\alpha+\beta-1}\int_\T
  u^{(\alpha+\beta-3)/2}\big(u^{(\alpha+\beta-1)/2}\big)_{xx}(u^\beta)_{xx}dx.
\end{align*}
We wish to estimate the second time derivative. To this end, we set
$w=u^{(\alpha+\beta-1)/2}$, $A=2\beta/(\alpha+\beta-1)$, 
and $B=(\alpha+\beta-3)/(\alpha+\beta-1)$. Then the derivatives can be written as
\begin{equation}
  \frac{dH_c}{dt}
	= -\frac{4\alpha\beta}{(\alpha+\beta-1)^2}\int_\T w_x^2 dx, \quad
	\frac{d^2H_c}{dt^2}
	= \frac{4\alpha\beta}{\alpha+\beta-1}\int_\T (w^{A})_{xx}w_{xx} w^B dx.
	\label{3.dHdt}
\end{equation}

In Lemma \ref{lem.ineq} below we show that there exists $\kappa_c>0$ such that
$$
  \int_\T (w^A)_{xx} w_{xx} w^B dx \ge \kappa_c\int_\T w^{A+B-1}w_{xx}^2 dx
$$
if the assumption $(2A-B-1)(A+B-2)<0$ holds. (Note that $\beta\neq 1$ is equivalent
to $A+B-2\neq 0$.) This condition is actually satisfied since
$$
  (2A-B-1)(A+B-2) = \frac{2(\alpha-\beta-2)(\alpha-\beta+1)}{(\alpha+\beta-1)^2} < 0,
$$
and we infer that
$$
  \frac{d^2H_c}{dt^2}
	\ge \frac{4\alpha\beta\kappa_c}{\alpha+\beta-1}\int_\T w^{A+B-1}w_{xx}^2 dx
	= \frac{4\alpha\beta\kappa_c}{\alpha+\beta-1}\int_\T u^{\beta-1}
	\big(u^{(\alpha+\beta-1)/2}\big)_{xx}^2 dx.
$$
Furthermore, by the Poincar\'e inequality applied to $w_x$ (see Lemma \ref{lem.poin}),
 \begin{align*}
  \int_\T & u^{\beta-1}	\big(u^{(\alpha+\beta-1)/2}\big)_{xx}^2 dx
	\ge \min_\T u^{\beta-1}\int_\T\big(u^{(\alpha+\beta-1)/2}\big)_{xx}^2 dx \\
  &\ge 4\pi^2\min_\T u^{\beta-1}
	\int_\T\big(u^{(\alpha+\beta-1)/2}\big)_{x}^2 dx 
	= 4\pi^2\min_\T u^{\beta-1}\int_\T w_x^2dx,
\end{align*}
and it follows that
\begin{equation}\label{3.Bakry}
  \frac{d^2H_c}{dt^2}
	\ge \frac{16\pi^2\alpha\beta\kappa_c}{\alpha+\beta-1}
	\min_\T u^{\beta-1}\int_\T w_x^2dx = -\lambda_c\frac{dH_c}{dt}.
\end{equation}
Denoting by $P_c=-dH_c/dt$ the entropy production, this inequality can be
formulated as $dP_c/dt\le -\lambda_c P_c$. Gronwall's lemma then implies that
$P_c(u(t))\le P_c(u_0)e^{-\lambda_c t}$ for $t>0$ and in particular
$\lim_{t\to\infty}P_c(u(t))=0$. 

Integrating \eqref{3.Bakry} over $(t,s)$ with $t<s$ 
and passing to the limit $s\to\infty$, we see that
$$
  \int_0^\infty\int_\T  w_x^2 dx \le H_c(u_0) < \infty.
$$
Thus, there exists a sequence $t_j\to\infty$ such that $\|w_x(t_j)\|_{L^2(\T)}\to 0$.
Following the arguments of \cite[Prop.~1ii]{CDGJ06},\footnote{Also see the erratum
http://www.asc.tuwien.ac.at/$\sim$juengel/publications/pdf/errata05carri.pdf.}
it follows that
$\lim_{t_j\to\infty}H_c(u(t_j))=0$, and since $t\mapsto H_c(u(t))$ is nonincreasing,
any sequence converges, $\lim_{t\to\infty}H_c(u(t))=0$.

We integrate \eqref{3.Bakry} over $(t,\infty)$ and use
$\lim_{t\to\infty}(dH_c/dt)(u(t))=\lim_{t\to\infty}H_c(u(t))=0$:
$$
  -\frac{dH_c}{dt}(u(t)) \ge \lambda_c H_c(u(t)), \quad t>0.
$$
Thus, another application of Gronwall's lemma gives the conclusion.
\end{proof}
  
It remains to prove Lemma \ref{lem.ineq}. Set
\begin{equation}\label{3.Rc}
  R_c = \{(A,B)\in\R^2:A>0,\ (2A-B-1)(A+B-2)<0\}.
\end{equation}
	
\begin{lemma}\label{lem.ineq}
Let $(A,B)\in R_c$. Then for all smooth positive functions $w$,
\begin{equation}\label{3.ineq}
  \int_\T w_{xx}(w^A)_{xx} w^B dx \ge \kappa_c\int_\T w_{xx}^2 w^{A+B-1}dx,
\end{equation}
where
\begin{equation}\label{kappac}
  \kappa_c = \left\{\begin{array}{ll}
	-A(2A-B-1)/(A+B-2)>0 & \quad\mbox{if }A+B-2\neq 0, \\
	A & \quad\mbox{if }A+B-2=0.
	\end{array}\right.
\end{equation}
\end{lemma}

\begin{proof}
The idea of the proof is to employ systematic integration by parts \cite{JuMa06}.
Since
\begin{equation}\label{3.ibpf}
  \int_\T(w_x^3 w^{A+B-2})_x dx = 0,
\end{equation}
we can formulate \eqref{3.ineq} as the following problem:
Find $c\in\R$ and $\kappa_c>0$ such that for all smooth positive functions $w$,
$$
  \int_\T\big(w_{xx}(w^A)_{xx} w^B + c(w_x^3 w^{A+B-3})_x 
	- \kappa_c w_{xx}^2 w^{A+B-1}\big)dx \ge 0.
$$
Calculating the derivatives and setting $\xi_1=w_x/w$, 
$\xi_2=w_{xx}/w$, this inequality is equivalent to
\begin{equation*}
  \int_\T w^{A+B-1}\Big((A-\kappa_c)\xi_2^2 + (A^2-A+3c)\xi_2\xi_1^2
	+ c(A+B-2)\xi_1^4\Big)dx \ge 0.
\end{equation*} 
The idea is to interpret the integrand as a polynomial in the variables
$\xi_1$, $\xi_2$ and to solve the following polynomial decision problem:
Find $c\in\R$ and $\kappa_c>0$ such that for all $(\xi_1,\xi_2)\in\R^2$,
\begin{equation}\label{3.aux}
  (A-\kappa_c)\xi_2^2 + (A^2-A+3c)\xi_2\xi_1^2
	+ c(A+B-2)\xi_1^4 \ge 0.
\end{equation}
This problem can be solved explicitly. Clearly, it must hold that
$A\ge \kappa_c>0$.
We distinguish two cases: $\kappa_c=A$ and $\kappa_c<A$.

First let $0<\kappa_c<A$. Then \eqref{3.aux} is valid if
the discriminant is nonpositive,
\begin{align*}
  0 & \ge (A^2-A+3c)^2 - 4c(A-\kappa_c)(A+B-2) \\
	&= \bigg(3c + A(A-1)-\frac23(A-\kappa_c)(A+B-2)\bigg)^2 \\
	&\phantom{xx}{}- \frac49(A-\kappa_c)^2(A+B-2)^2 
	+ \frac43A(A-1)(A-\kappa_c)(A+B-2).
\end{align*}
Choosing the minimizing value
\begin{align}
  c &= -\frac13\bigg(A(A-1)-\frac23(A-\kappa_c)(A+B-2)\bigg) \nonumber \\
	&= -\frac{A}{9}(A-2B+1) - \frac29\kappa_c(A+B-2), \label{3.c}
\end{align}
the discriminant is nonpositive if and only if
\begin{align*}
  0 &\ge -\frac49(A-\kappa_c)^2(A+B-2)^2 + \frac43A(A-1)(A-\kappa_c)(A+B-2) \\
	&= \frac49(A-\kappa_c)(A+B-2)\big(\kappa_c(A+B-2) + A(2A-B-1)\big).
\end{align*}
Set $\kappa_c=\eps A$ for $0<\eps<1$. Then the 
previous inequality is true if and only if
\begin{equation}\label{3.aux2}
  A(A+B-2)\big(\eps(A+B-2)+2A-B-1\big) \le 0.
\end{equation}
We infer that if
$$
  \eps = -\frac{2A-B-1}{A+B-2} > 0
$$
then \eqref{3.aux} holds. This implies that $\kappa_c = \eps A=-A(2A-B-1)/(A+B-2)>0$
and we need to choose $A>0$ and $(2A-B-1)(A+B-2)<0$.

Next, let $\kappa_c=A$. Then the quadratic term in $\xi_2$ in \eqref{3.aux} 
vanishes and the mixed term must vanish too, i.e.\ $c=-A(A-1)/3$.
Hence, the coefficient of the remaining term has to be nonnegative,
i.e.\ $-A(A-1)(A+B-2)\ge 0$. If $A=1$, inequality \eqref{3.ineq} becomes trivial.
The set of all $(A,B)$ such that $A>0$ and
$(A-1)(A+B-2) < 0$ is contained in the set of all $(A,B)$ satisfying $A>0$ and
$(2A-B-1)(A+B-2)<0$. This finishes the proof.
\end{proof}

We state now a discrete version of inequality \eqref{3.ineq}.

\begin{lemma}\label{lem.ineqd}
Let $w_0,\ldots,w_{N+1}\in\R$ satisfy $w_N=w_0$, $w_{N+1}=w_1$ and let $0<\eps\le 1$.
There exists a region $R\subset\R^2$, 
containing the line $A=1$, such that for all $(A,B)\in R$,
\begin{align}
  \sum_{i=1}^{N} & (w_{i+1}-2w_i+w_{i-1})(w_{i+1}^A-2w_i^A+w_{i-1}^A)w_i^B 
	\nonumber \\
	&\ge \kappa\sum_{i=1}^{N} \min_{j=i,i\pm 1}w_j^{A+B-1}(w_{i+1}-2w_i+w_{i-1})^2,
	\label{3.ineqd}
\end{align}
where $\kappa = \eps A>0$.
\end{lemma}

The lemma is trivial as stated since \eqref{3.ineqd} clearly holds for 
$R=\{(A,B):A=1\}$ with $\kappa=1$. 
Figure \ref{fig.admregion} ilustrates the numerical
admissible regions for $(A,B)$ for two different values of $\eps$.
The admissible region $R$ is smaller than the region $R_c$ for the continuous case
but it approaches the latter region as $\kappa\to 0$.

\begin{figure}[ht]
\includegraphics[width=80mm]{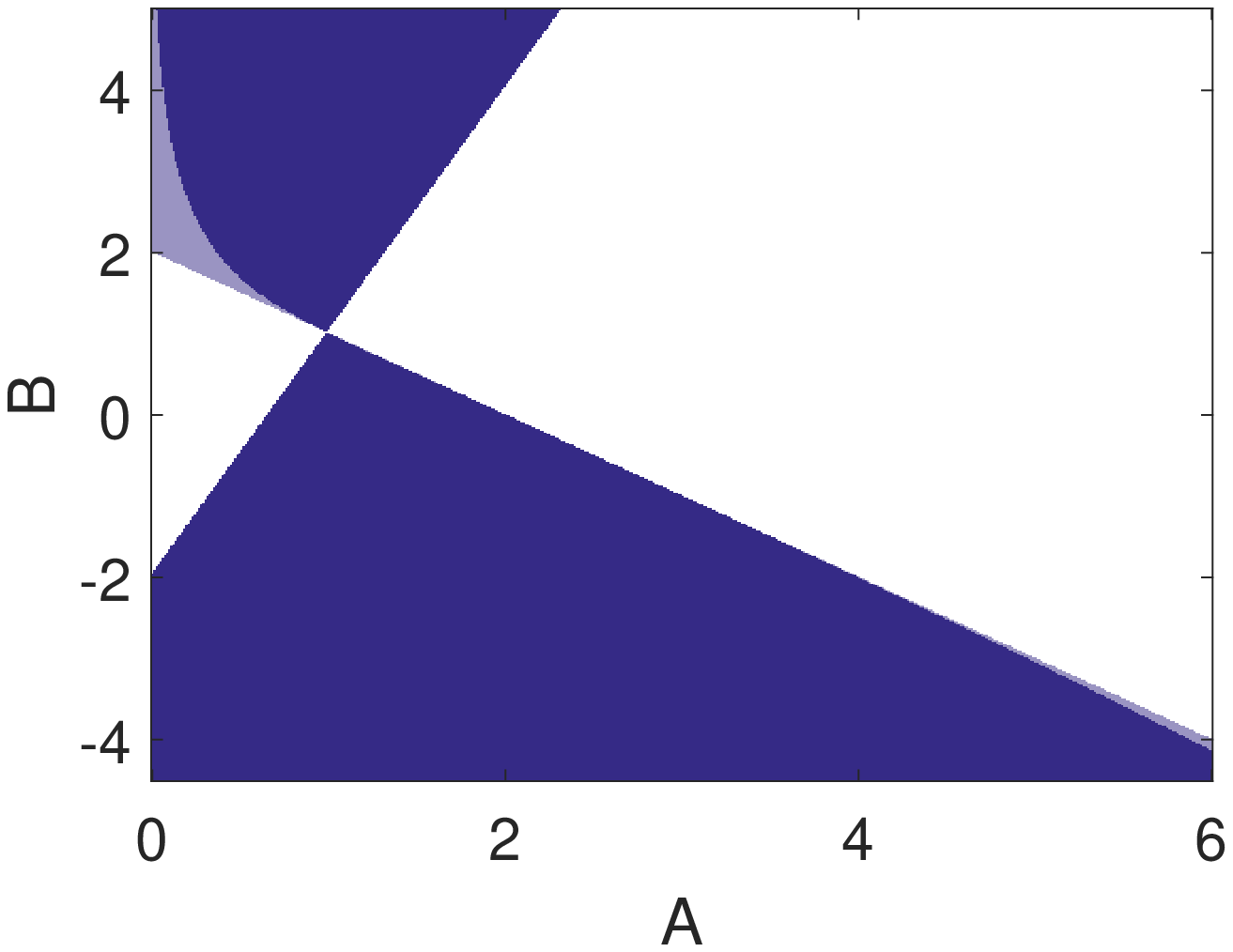}
\includegraphics[width=80mm]{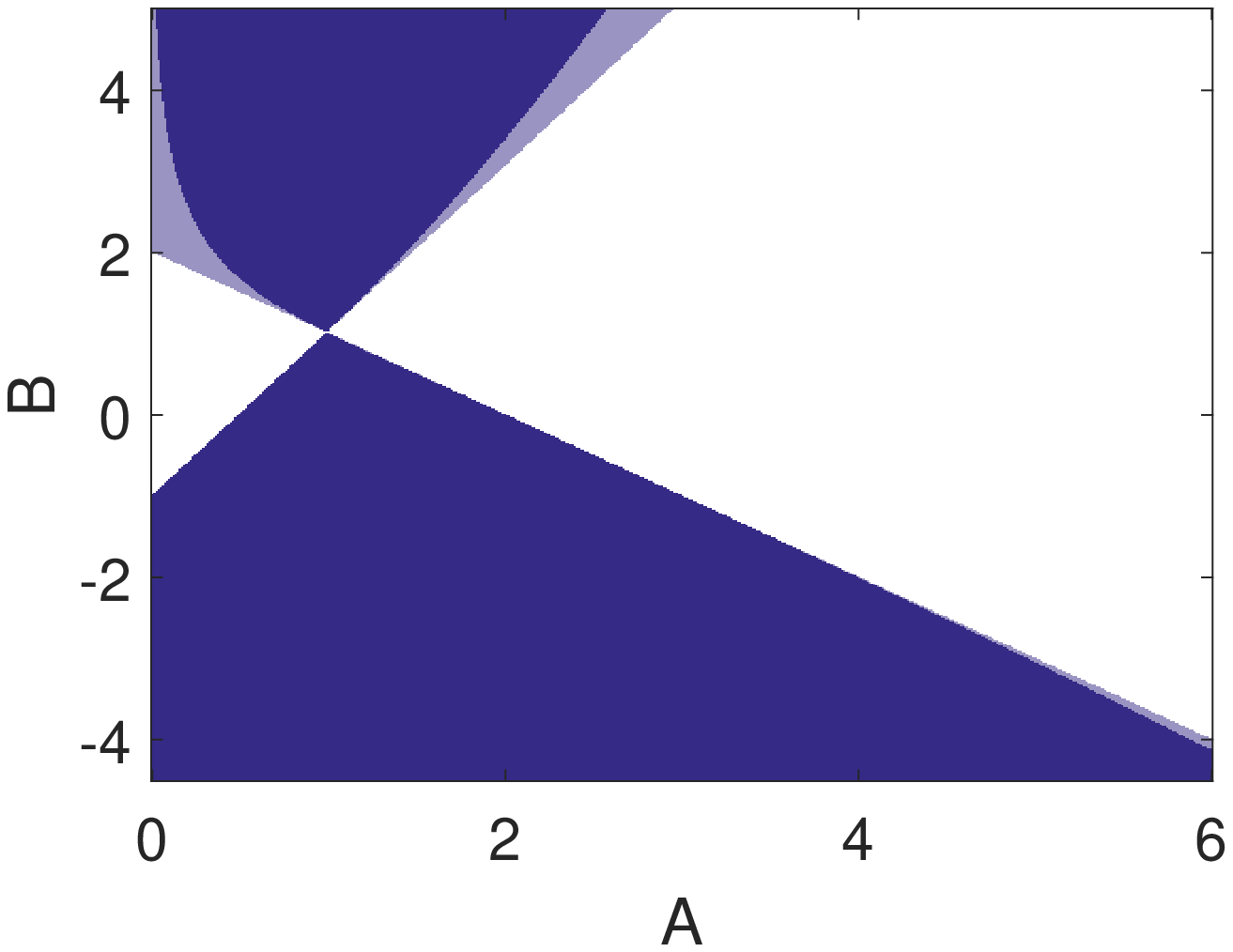}
\caption{The regions of admissible $(A,B)$ such that $T(X,Y)\ge 0$ for all 
$X$, $Y\ge 0$ using $c$ as in \eqref{3.c} with $\kappa_c=\kappa$ 
and $\kappa=A/4$ (left), 
$\kappa=A/100$ (right). The set $R$ is depicted in dark blue, $R_c\supset R$ 
in light blue.}
\label{fig.admregion}
\end{figure}

The idea of the proof of \eqref{3.ineqd}
is to add the following discrete version of the integration-by-parts 
formula \eqref{3.ibpf}, namely
$$
  \frac{1}{\rho^3}\sum_{i=1}^{N} \Big(M(w_{i+1},w_i)^{A+B+1-3\rho}
	(w_{i+1}^\rho-w_i^\rho)^3 
	- M(w_{i},w_{i-1})^{A+B+1-3\rho}(w_{i}^\rho-w_{i-1}^\rho)^3\Big) = 0.
$$
The sum vanishes because of the periodic boundary conditions. 
Here $\rho>0$ is a free parameter, and the function
$M(x,y)$ is a symmetric mean value, i.e., it satisfies
\begin{equation}\label{3.M}
  M(x,y)=M(y,x), \quad M(\lambda x,\lambda y)=\lambda M(x,y), \quad M(x,x)=x
\end{equation}
for all $x$, $y$, $\lambda\ge 0$. For the numerical simulations below, we will choose 
$\rho=(A+B+1)/3$ such that the mean function does not need to be specified.
Then \eqref{3.ineqd} holds if we can show the following inequality 
for all admissible $(A,B)$ and $w_i\neq 0$:
\begin{align*}
  &\sum_{i=1}^{N} w_i^{A+B+1}\bigg\{
	\bigg(\bigg(\frac{w_{i+1}}{w_i}\bigg)^A + \bigg(\frac{w_{i-1}}{w_i}\bigg)^A - 2\bigg)
	\bigg(\frac{w_{i+1}}{w_i} + \frac{w_{i-1}}{w_i} - 2\bigg) \\
	&\phantom{x}{}- \kappa\min_{j=i,i\pm 1}\bigg(\frac{w_j}{w_i}\bigg)^{A+B-1}
	\bigg(\frac{w_{i+1}}{w_i} + \frac{w_{i-1}}{w_i} - 2\bigg)^2 \\
	&\phantom{x}{}+ \frac{c}{\rho^3}\bigg(M\bigg(\frac{w_{i+1}}{w_i},1\bigg)^{A+B+1-3\rho}
	\bigg(\left(\frac{w_{i+1}}{w_i}\right)^\rho-1\bigg)^3\\
	&\phantom{xxxxxx}- M\bigg(\frac{w_{i-1}}{w_i},1\bigg)^{A+B+1-3\rho}
	\bigg(1-\left(\frac{w_{i-1}}{w_i}\right)^\rho\bigg)^3
	\bigg)\bigg\} \ge 0.
\end{align*}

We verify this inequality pointwise, i.e.\ setting $X=w_{i+1}/w_i$ and
$Y=w_{i-1}/w_i$, we wish to find $c\in\R$, $\kappa>0$ such that
for all $X$, $Y>0$,
\begin{align}
  T(X,Y) &:= (X^A+Y^A-2)(X+Y-2) \nonumber \\
	&\phantom{xx}{}+ \frac{c}{\rho^3}\Big(M(X,1)^{A+B+1-3\rho}(X^\rho-1)^3 
	+ M(Y,1)^{A+B+1-3\rho}(Y^\rho-1)^3\Big) \label{3.T} \\
	&\phantom{xx}{}- \kappa\min\{1,X^{A+B-1},Y^{A+B-1}\}(X+Y-2)^2 \ge 0. \nonumber
\end{align}
The first term $(X^A+Y^A-2)(X+Y-2)$ becomes negative in certain regions;
see Figure \ref{fig.levelset}. It is compensated by the second term (shift term)
on the right-hand side of \eqref{3.T} if we choose the constant $c$ according to 
\eqref{3.c} with $\kappa=\kappa_c$ as in \eqref{kappac}.

\begin{figure}[ht]
\includegraphics[width=75mm]{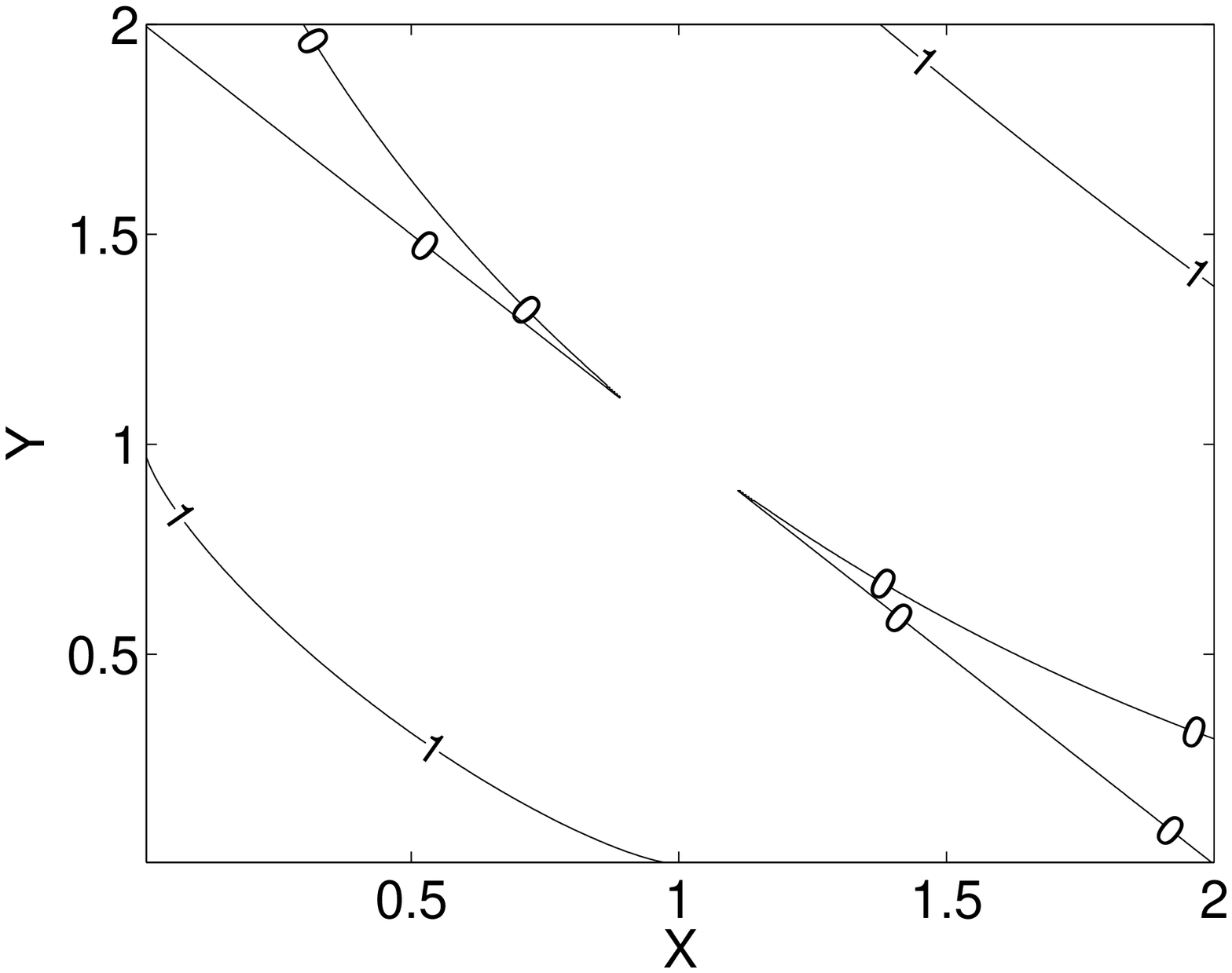}
\includegraphics[width=75mm]{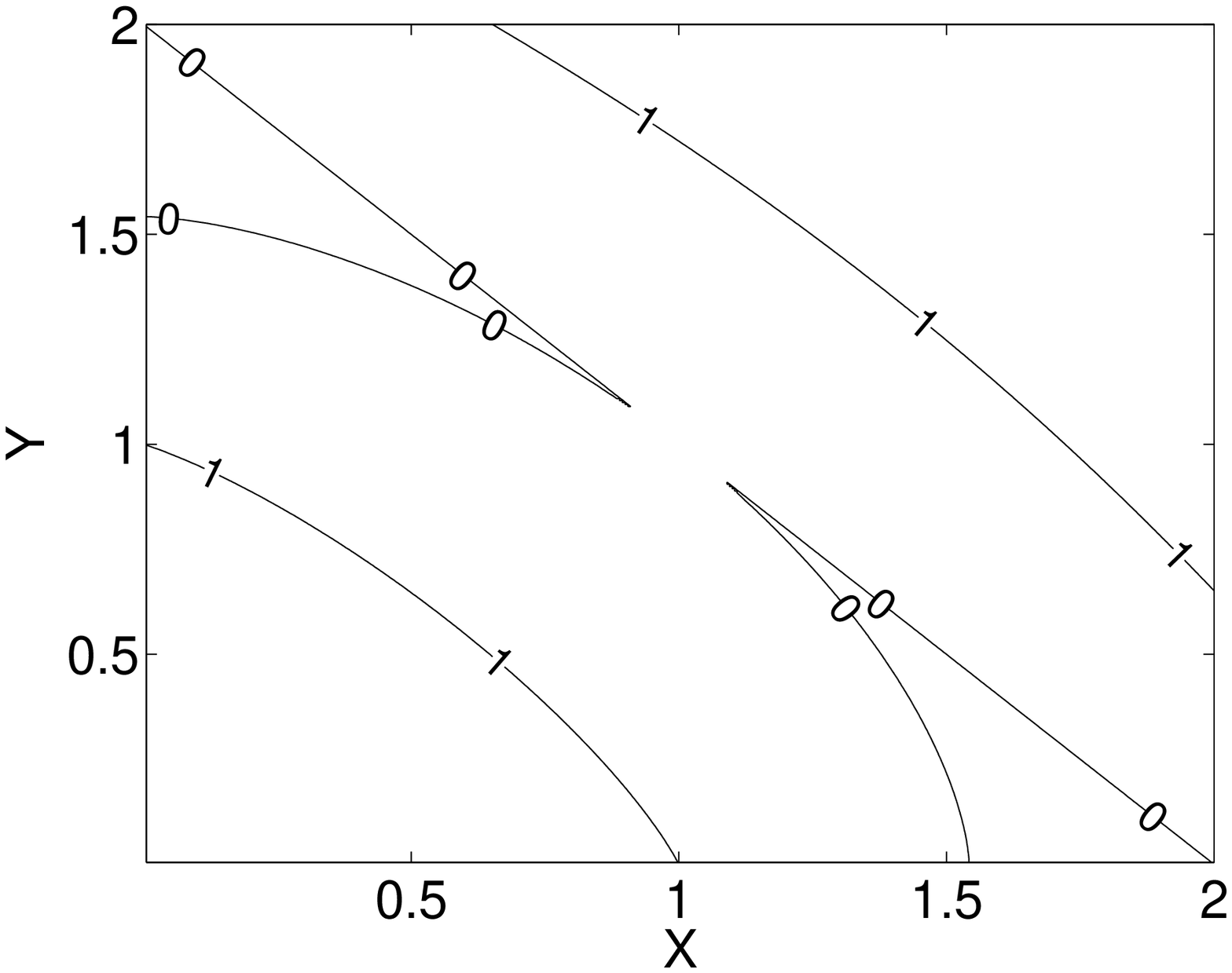}
\caption{Level sets $(X^A+Y^A-2)(X+Y-2)=0$ and $(X^A+Y^A-2)(X+Y-2)=1$ for 
$A=0.6$, $B=4$ (left) and
$A=1.6$, $B=2.5$ (right). We have chosen $\kappa=\kappa_0=A/200$ and $c$ as in 
\eqref{3.c}.}
\label{fig.levelset}
\end{figure}

Unfortunately, it seems to be difficult to prove \eqref{3.T} analytically
in full generality.
Note that polynomial quantifier elimination does not apply if $A$ and $B$
are not integers, since the function $T(X,Y)$ generally is {\em not} a
polynomial. Instead, we verify \eqref{3.T} analytically for all 
$(A,B)\in R_c$ and all $(X,Y)$ in some neighborhood of $(1,1)$. 

\begin{lemma}\label{lem.T}
Let $T$ be given by \eqref{3.T} and let $(A,B)\in R_c$, where $R_c$ is
defined in \eqref{3.Rc}. 
Then there exists a neighborhood $W$ of $(1,1)$ such that for all $(X,Y)\in W$, 
$$
  T(X,Y) \ge 0
$$
holds for $c$ as in \eqref{3.c} and with $\kappa_c=\kappa$ as in \eqref{kappac}.
\end{lemma}

If the step size $h>0$ is small enough, we expect that the quotients
$w_{i+1}/w_i$ are close to one for all $i=0,\ldots,N-1$. This means that 
$(X,Y)$ lies in a neighborhood of $(1,1)$, and the lemma applies.

\begin{proof}
We use the local coordinates $u=(X+Y-2)/h^2$ and $v=(X-Y)/(2h)$, which correspond 
to (central) second-order and first-order derivatives. Then 
$X = 1+hv+h^2 u/2$ and $Y=1-hv+h^2 u/2$. We develop $T$ as a function of $h$
at $h=0$. For this, we observe that $M(1,1)=1$ and $M_X(1,1)=M_Y(1,1)=1/2$.
Indeed, we infer from the properties \eqref{3.M} that
\begin{align*}
  M_X(1,1) &= \lim_{\eps\to 0}\frac{1}{\eps}\big(M(1+\eps,1)-M(1,1)\big)
	= \lim_{\eps\to 0}\bigg(\frac{1+\eps}{\eps}M\bigg(1,\frac{1}{1+\eps}\bigg) 
	- \frac{1}{\eps}\bigg) \\
	&= \lim_{\eps\to 0}\bigg(\frac{1+\eps}{\eps}M\bigg(1,1-\frac{\eps}{1+\eps}\bigg)
	- \frac{1+\eps}{\eps} + 1\bigg) \\
	&= \lim_{\eps\to 0}\bigg(\frac{1+\eps}{\eps}M\bigg(1-\frac{\eps}{1+\eps},1\bigg)
	- \frac{1+\eps}{\eps} + 1\bigg) \\
	&= \lim_{\delta\to 0}\frac{1}{\delta}\big(M(1-\delta,1)-M(1,1)\big) + 1 \\
	&= -M_X(1,1) + 1,
\end{align*}
which implies that $M_X(1,1)=1/2$.

Calculating the Taylor series of the terms in $T$ with respect to $h$ at $h=0$ 
leads to
\begin{align*}
  (X^A+Y^A-2)(X+Y-2) &= Au\big((A-1)v^2+u\big)h^4 + O(h^6), \\
	\frac{c}{\rho^3}M(X,1)^{A+B+1-3\rho}(X^\rho-1)^3 
	&= cv^3h^3 + \frac{c}{2}\big((A+B-2)v^2+3u\big)v^2 h^4 
	+ O(h^5), \\
	\frac{c}{\rho^3}M(Y,1)^{A+B+1-3\rho}(Y^\rho-1)^3 
	&= -cv^3h^3 + \frac{c}{2}\big((A+B-2)v^2+3u\big)v^2 h^4 
	+ O(h^5).
\end{align*}
In particular, as expected, the explicit choices of both $\rho$ and $M(x,y)$ 
do not change the behavior of the shift term locally around the equilibrium 
$w_{i-1}=w_{i}=w_{i+1}$ or $h=0$.
Moreover, $\min\{1,X^{A+B-1},Y^{A+B-1}\}=1+O(h)$ and $(X+Y-2)^2=u^2h^4$.
Combining these expressions gives
$$
  T(X,Y) = h^4\Big((A-\kappa)u^2 + \big(A(A-1)+3c\big)uv^2 + c(A+B-2)v^4\Big) + O(h^5).
$$
The polynomial
$$
  (u,v)\mapsto (A-\kappa)u^2 + \big(A(A-1)+3c\big)uv^2 + c(A+B-2)v^4
$$
is the same as in \eqref{3.aux}. The proof of Lemma \ref{lem.ineq} shows that
it is nonnegative for all $(A,B)\in R_c$ with $c$ as in \eqref{3.c} and $\kappa_c$
as in \eqref{kappac}. We deduce that
$T(X,Y)\ge 0$ holds for all $(A,B)\in R_c$ if $h\in\R$ is sufficiently small.
This proves the lemma.
\end{proof}

%%%%%%%%%%%%%%%%%%%%%%%%%%%%%%%%%%%%%%%%%%%%%%%%%%%%%%%%%%%%%%%%%%%%%%%%%%%%%

\section{Discrete porous-medium equation}\label{sec.pme}

We apply the abstract Bakry-Emery method to a finite-difference approximation
of the porous-medium equation, i.e., we choose $A(u)=-(u^\beta)_{xx}$ on $\T$ for
suitable functions $u$. 
Let $\tau>0$ be the time step and $h>0$ the space step. A natural scheme would be
$$
  u_i^k - u_i^{k-1} = \tau h^{-2}\big((u_{i+1}^k)^\beta - 2(u_i^k)^\beta
	+ (u_{i-1}^k)^\beta\big),
$$
for all $i=1,\ldots,N$, $k\in\N$, and $u_N^k=u_0^k$, $u_{N+1}^k=u_1^k$.
The corresponding discrete entropy is 
$H(u^k)=h\sum_{i=1}^N ((u^k_i)^\alpha-\overline{u}^\alpha)/(\alpha-1)$ and
$\overline{u}=h\sum_{i=1}^N u_i^0$ is the constant steady state. We choose $\alpha>1$
and $\beta>1$.

Unfortunately, 
the abstract Bakry-Emery method cannot be applied to this scheme. 
The problem is the second inequality in Assumption A1. Indeed, using
the inequality $y^\alpha-z^\alpha\ge \alpha z^{\alpha-1}(y-z)$ for all $y$, $z\ge 0$,
which follows from the convexity of $z\mapsto z^\alpha$ for $\alpha>1$,
inserting the numerical scheme and then summing by parts, we find that
\begin{align*}
  -\tau P &= H(u^k) - H(u^{k-1})
	= h\sum_{i=0}^N\big((u_i^k)^\alpha - (u_i^{k-1})^\alpha\big) \\
	&\ge \alpha h\sum_{i=1}^N(u_i^{k-1})^{\alpha-1}(u_i^k-u_i^{k-1}) \\
	&= \alpha h^{-1}\tau\sum_{i=1}^N(u_i^{k-1})^{\alpha-1}
	\Big(\big((u_{i+1}^k)^\beta-(u_i^k)^\beta\big) 
	- \big((u_{i}^k)^\beta-(u_{i-1}^k)^\beta\big)\Big) \\
	&= -\alpha h^{-1}\sum_{i=1}^N
	\big((u_{i+1}^{k-1})^{\alpha-1}-(u_i^{k-1})^{\alpha-1}\big)
	\big((u_{i+1}^k)^{\beta}-(u_i^k)^{\beta}\big).
\end{align*}
This expression cannot be estimated further; it may even have the wrong
sign. We need a scheme that avoids the use of the inequality 
$y^\alpha-z^\alpha\ge \alpha z^{\alpha-1}(y-z)$. We stress the fact that this
problem does not occur in the semi-discrete scheme
$$
  \pa_t u_i = h^{-2}\big((u_{i+1}^k)^\beta - 2(u_i^k)^\beta
	+ (u_{i-1}^k)^\beta\big),
$$
since then 
\begin{align*}
  \frac{dH}{dt} &= \frac{\alpha h}{\alpha-1}\sum_{i=0}^N u_i^{\alpha-1}\pa_t u_i
  = \frac{\alpha}{(\alpha-1)h}\sum_{i=0}^N u_i^{\alpha-1}
  \big((u_{i+1}^k)^\beta - 2(u_i^k)^\beta	+ (u_{i-1}^k)^\beta\big) \\
	&= -\frac{\alpha}{(\alpha-1)h}\sum_{i=0}^N
	\big((u_{i+1}^{k})^{\alpha-1}-(u_i^{k})^{\alpha-1}\big)
	\big((u_{i+1}^k)^{\beta}-(u_i^k)^{\beta}\big),
\end{align*}
and this expression is nonpositive (since $\alpha>1$).

Our idea is to make the entropy production {\em linear} in its argument. For this, we
introduce the new variable $v_i^k=(u_i^k)^\alpha$. 
In the (continuous) variable $v=u^\alpha$, the evolution equation transforms to
$\pa_t v = -v^{(\alpha-1)/\alpha}\Delta(v^{\beta/\alpha})$, 
which inspires the numerical scheme
\begin{equation}\label{4.scheme}
  v_i^k - v_i^{k-1} = \alpha\tau h^{-2}(v_i^k)^{(\alpha-1)/\alpha}
	\big((v_{i+1}^k)^{\beta/\alpha} - 2(v_i^k)^{\beta/\alpha} 
	+ (v_{i-1}^k)^{\beta/\alpha}\big),
\end{equation}
for all $i=1,\ldots,N$, $k\in\N$, and $v_N^k=v_0^k$, $v_{N+1}^k=v_1^k$. 
The discrete entropy and Fisher information become
$$
  H(v^k) = \frac{h}{\alpha-1}\sum_{i=1}^N(v_i^k - V),
	\quad F(v^k) = \frac{1}{h}\sum_{i=1}^{N}
	\big((v_{i+1}^k)^\gamma - (v_i^k)^\gamma\big)^2,
$$
where $V>0$ has to be determined and $\gamma=(\alpha+\beta-1)/(2\alpha)$.
The entropy production can be estimated, using summation by parts, as
\begin{align}
  -\tau P(v^k) &= H(v^k)-H(v^{k-1})
	= \frac{h}{\alpha-1}\sum_{i=1}^N(v^k_i-v_i^{k-1}) \nonumber \\
	&= \frac{\alpha\tau}{(\alpha-1)h}\sum_{i=1}^N(v_i^k)^{(\alpha-1)/\alpha}
	\big((v_{i+1}^k)^{\beta/\alpha} - 2(v_i^k)^{\beta/\alpha} 
	+ (v_{i-1}^k)^{\beta/\alpha}\big) \nonumber \\
	&= -\frac{\alpha\tau}{(\alpha-1)h}\sum_{i=1}^N\big((v_{i+1}^k)^{(\alpha-1)/\alpha}-
	(v_i^k)^{(\alpha-1)/\alpha}\big)\big((v_{i+1}^k)^{\beta/\alpha} 
	- (v_i^k)^{\beta/\alpha}\big) \le 0. \label{4.P}
\end{align}
According to Lemma \ref{lem.xy}, the entropy production can be estimated 
from below {\em and} above in terms of the Fisher information. 

After this motivation, we prove the existence of solutions to \eqref{4.scheme}. 

\begin{lemma}\label{lem.ex}
For given $v_i^{k-1}\ge 0$, $i=1,\ldots,N$, there exists a solution
$v_i^k\ge 0$, $i=1,\ldots,N$, to \eqref{4.scheme}.
\end{lemma}

\begin{proof}
We give only a sketch of the proof since the existence of solutions follows from
a standard fixed-point theorem. We only provide the a priori estimates needed
for this argument. First multiply \eqref{4.scheme} by $(v_i^k)_- = \min\{v_i^k,0\}$
and sum over $i=1,\ldots,N$. Since $v_i^k(v_i^k)_-=(v_i^k)^2_-$, we obtain
\begin{align*}
  \sum_{i=1}^N (v_i^k)_-^2
	&= \sum_{i=1}^N v_i^{k-1}(v_i^k)_-
	+ \alpha\tau h^{-2}\sum_{i=1}^N (v_i^k)^{2-1/\alpha}_-\big((v_{i+1}^k)^{\beta/\alpha} 
	- 2(v_i^k)^{\beta/\alpha} + (v_{i-1}^k)^{\beta/\alpha}\big) \\
	&\le \alpha\tau h^{-2}\sum_{i=1}^N (v_i^k)^{2-1/\alpha}_-
	\big(((v_{i+1}^k)^{\beta/\alpha} - (v_i^k)^{\beta/\alpha}) 
	- ((v_i^k)^{\beta/\alpha} - (v_{i-1}^k)^{\beta/\alpha})\big).
\end{align*}
By summation by parts, this becomes
$$
  \sum_{i=1}^N (v_i^k)_-^2 \le -\alpha\tau h^{-2}\sum_{i=1}^N 
	\big((v_{i+1}^k)^{2-1/\alpha}_- - (v_i^k)^{2-1/\alpha}_-\big)
	\big((v_{i+1}^k)^{\beta/\alpha} - (v_i^k)^{\beta/\alpha}\big) \le 0,
$$
since $z\mapsto z_-^{2-1/\alpha}$ is nondecreasing. 
We infer that $(v_i^k)_-=0$ and hence $v_i^k\ge 0$.
Next, by \eqref{4.P},
$$
  \sum_{i=1}^N v_i^k \le \sum_{i=1}^N v_i^{k-1} \le \sum_{i=1}^N v_i^0,
$$
and this is the desired a priori estimate.
\end{proof}

Next, we turn to the proof of our main result.

\begin{proof}[Proof of Theorem \ref{thm.main}]
We verify the assumptions of Proposition \ref{prop.BE}. 
For this, we continue our estimates for $P$.  Applying 
Lemma \ref{lem.xy} with $a=(\alpha-1)/\alpha$ and $b=\beta/\alpha$ to \eqref{4.P},
we obtain the inequalities
$$
  P(v^k) \le \frac{\alpha}{\alpha-1} F(v^k), \quad
	F(v^k)\le \frac{\alpha\gamma^2}{\beta} P(v^k) 
	= \frac{(\alpha+\beta-1)^2}{4\alpha\beta}P(v^k).
$$
Thus, Assumption A1 is satisfied with $C_m=4\alpha\beta/(\alpha+\beta-1)^2$ 
and $C_M=\alpha/(\alpha-1)$.

Next, we estimate the difference $F(v^k)-F(v^{k-1})$. To this end, we set
$v_i:=v_i^k$, $\overline{v}_i:=v_i^{k-1}$, $a_i:=(v_i-\overline{v}_i)/\tau$ and write
\begin{align*}
  F(v)-F(\overline{v})
	&= \frac{1}{h}\sum_{i=1}^N \big((v_{i+1}^\gamma-v_i^\gamma)^2
	- (\overline{v}_{i+1}^\gamma-\overline{v}_i^\gamma)^2\big) \\
	&= \frac{1}{h}\sum_{i=1}^N\Big((v_{i+1}^\gamma-v_i^\gamma)^2
	- \big((v_{i+1}-\tau a_{i+1})^\gamma-(v_i-\tau a_i)^\gamma\big)^2\Big) \\
	&=: G(\tau).
\end{align*}
The idea of the proof is to expand $G(\tau)$ around zero:
$$
  F(v)-F(\overline{v}) = G(0) + G'(0)\tau + \frac12 G''(\xi)\tau^2
$$
for some $\xi\in(0,\tau)$. We show that the right-hand side can be bounded from
above by $-\tau KF(v)$ for some $K>0$, 
which verifies Assumption A2. This idea has been first employed in \cite{CJS16}. 
Clearly, we have $G(0)=0$. The first derivatives of $G$ equal
\begin{align*}
  G'(\tau) &= 2\gamma h^{-1}\sum_{i=1}^N
	\big((v_{i+1}-\tau a_{i+1})^\gamma-(v_i-\tau a_i)^\gamma\big) \\
	&\phantom{xx}{}\times
	\big((v_{i+1}-\tau a_{i+1})^{\gamma-1}a_{i+1}-(v_i-\tau a_i)^{\gamma-1}a_i\big), \\
  G''(\tau) &= -2\gamma h^{-1}\sum_{i=1}^N\Big(
	\gamma\big((v_{i+1}-\tau a_{i+1})^{\gamma-1}a_{i+1}
	-(v_i-\tau a_i)^{\gamma-1}a_i\big)^2 \\
	&\phantom{=-2\gamma h^{-1}\sum_{i=1}^N\Big(}{}+ (\gamma-1)
	\big((v_{i+1}-\tau a_{i+1})^\gamma-(v_i-\tau a_i)^\gamma\big) \\
	&\phantom{xx}{}\times
	\big((v_{i+1}-\tau a_{i+1})^{\gamma-2}a_{i+1}^2-(v_i-\tau a_i)^{\gamma-2}a_i^2\big)
	\Big).
\end{align*}

First, we claim that $G''(\tau)\le 0$ for any $\tau>0$. 
Indeed, we replace $v_i-\tau a_i$ by $\overline v_i$ and obtain
\begin{align*}
  G''(\tau) &= -2\gamma h^{-1}\sum_{i=0}^N(c_1 a_{i+1}^2 + c_2 a_{i+1}a_i + c_3 a_i^2), 
	\quad\mbox{where} \\
	c_1 &= \gamma\overline v_{i+1}^{2\gamma-2}+(\gamma-1)\overline v_{i+1}^{\gamma-2}
	(\overline v_{i+1}^\gamma-\overline v_i^\gamma)
	= (2\gamma-1)\overline v_{i+1}^{2\gamma-2} 
	- (\gamma-1)\overline v_{i+1}^{\gamma-2}\overline v_i^\gamma, \\
	c_2 &= -2\gamma\overline v_{i+1}^{\gamma-1}\overline v_i^{\gamma-1}, \\
	c_3 &= \gamma\overline v_i^{2\gamma-2}-(\gamma-1)\overline v_i^{\gamma-2}
	(\overline v_{i+1}^{\gamma}-\overline v_i^\gamma)
	= (2\gamma-1)\overline v_i^{2\gamma-2} 
	- (\gamma-1)\overline v_i^{\gamma-2}\overline v_{i+1}^\gamma.
\end{align*}
It holds that $c_1\ge 0$, since this inequality is equivalent to 
$(2\gamma-1)\overline v_{i+1}^\gamma\ge(\gamma-1)\overline v_i^\gamma$, 
and this is true
for $1/2\le\gamma\le 1$ (which is equivalent to $\beta\ge 1$ and 
$\alpha-\beta\ge -1$).
Moreover, the discriminant $4c_1c_3-c_2^2\ge 0$ is equivalent to
$$
  4(2\gamma-1)(1-\gamma)(\overline v_{i+1}\overline v_i)^{\gamma-2}
	(\overline v_{i+1}^\gamma-\overline v_i^\gamma)^2 \ge 0,
$$
which also holds true for $1/2\le\gamma\le 1$. This shows that $G''(\tau)\le 0$
and consequently,
$$
  F(v)-F(\overline v) = G(\tau)=G(0)+\tau G'(0)+\frac{\tau^2}{2}G''(\xi)
	\le \tau G'(0).
$$

It remains to compute $G'(0)$. Inserting the definition of $a_i$, we find that
\begin{align*}
  G'(0) &= 2\gamma h^{-1}\sum_{i=1}^N(v_{i+1}^\gamma-v_i^\gamma)
	(v_{i+1}^{\gamma-1}a_{i+1}-v_i^{\gamma-1}a_i) \\
	&= -2\gamma h^{-1}\sum_{i=1}^N v_i^{\gamma-1}a_i
	(v_{i+1}^{\gamma}-2v_i^{\gamma}+v_{i-1}^\gamma) \\
	&= -2\alpha\gamma h^{-3}\sum_{i=1}^N v_i^{(\alpha+\beta-3)/(2\alpha)}
	(v_{i+1}^{\beta/\alpha} - 2v_i^{\beta/\alpha} + v_{i-1}^{\beta/\alpha})
	(v_{i+1}^{\gamma}-2v_i^{\gamma}+v_{i-1}^\gamma),
\end{align*}
since $v_i^{\gamma-1}v_i^{(\alpha-1)/\alpha}=v_i^{(\alpha+\beta-3)/(2\alpha)}$.

We apply Lemma \ref{lem.ineqd} with $w_i=v_i^{\gamma}$, 
$A=2\beta/(\alpha+\beta-1)$ and $B=(\alpha+\beta-3)/(\alpha+\beta-1)$ and
infer that 
$$
  G'(0) \le -2\alpha\gamma\kappa h^{-3}\sum_{i=1}^N
	\min_{j=i,i\pm 1}v_j^{(\beta-1)/\alpha}
	(v_{i+1}^\gamma-2v_i^\gamma+v_{i-1}^\gamma)^2.
$$
By the discrete Poincar\'e-Wirtinger inequality (Lemma \ref{lem.PWI}), applied
with $z_i=v_{i+1}^\gamma-v_i^\gamma$, it follows that
\begin{align*}
  G'(0) &\le -2C_p^{-1}\alpha\gamma\kappa h^{-1}
	\min_{i=1,\ldots,N} v_i^{(\beta-1)/\alpha}
	\sum_{i=1}^N(v_{i+1}^\gamma-v_i^\gamma)^2 \\ 
	&= -2C_p^{-1}\alpha\gamma\kappa\min_{i=1,\ldots,N} 
	v_i^{(\beta-1)/\alpha}F(v),
\end{align*}
and hence, with $\kappa_0 = 2C_p^{-1}\alpha\gamma\kappa\min_{i=1,\ldots,N} 
v_i^{(\beta-1)/\alpha}$,
$$
  F(v) - F(\overline v) \le -\tau \kappa_0 F(v).
$$
This shows Assumption A2 of Proposition \ref{prop.BE} and, in particular,
after applying Gronwall's lemma, $\lim_{k\to\infty}F(v^k)=0$.

It remains to prove that Assumption A3, i.e.\ $\lim_{k\to\infty}H(v^k)=0$, holds.
We know that
$$
  v_i^k \le \sum_{j=1}^N v_j^k \le \sum_{j=1}^N v_j^0 < \infty,
$$
so, for any fixed $i=1,\ldots,N$, $(v_i^k)$ is bounded. Therefore, there
exists a sequence $k_j\to\infty$ such that $v_i^{k_j}\to y_i$ for some $y_i\ge 0$. 
By the discrete Poincar\'e-Wirtinger inequality (Lemma \ref{lem.PWI}), applied
to $z_i=(v_i^k)^\gamma - (V^k)^\gamma$, where 
$(V^k)^\gamma:=h\sum_{i=1}^{N}(v_i^k)^\gamma$, it follows that
$$
  \sum_{i=1}^N\big((v_i^k)^\gamma - (V^k)^\gamma\big)^2
	\le C_ph^{-2}\sum_{i=1}^N\big((v_{i+1}^k)^\gamma - (v_i^k)^\gamma\big)^2
	= C_ph^{-1} F(v^k).
$$
Since $\lim_{k\to\infty}F(v^k)=0$, we deduce that $(v_i^{k_j})$ and 
$V^{k_j}$ have the same limit, say $y:=y_i$. Set $U:=y^{1/\alpha}$.
This defines the entropy
$$
  \ent(u^k) = \frac{h}{\alpha-1}\sum_{i=1}^N\big((u_i^k)^\alpha - U^\alpha\big)
$$
for $u_i^k:=(v_i^k)^{1/\alpha}$. It holds that $\ent(u^{k_j})\to 0$ as $k_j\to\infty$.
But $k\mapsto \ent(u^k)$ is nonincreasing, from which we deduce that
$H(v^k)=\ent(u^k)\to 0$ for any sequence $k\to\infty$. 

According to Proposition \ref{prop.BE}, the discrete entropy converges
exponentially with decay rate
\begin{align*}
  \lambda 
	&= \frac{C_m}{C_M}\kappa_0 
	= \frac{4(\alpha-1)\beta}{\alpha+\beta-1}\frac{\kappa}{C_p}
	\min_{i=1,\ldots,N}u_i^{\beta-1} 
	= \frac{8\eps(\alpha-1)\beta^2}{C_p(\alpha+\beta-1)^2}
	\min_{i=1,\ldots,N}u_i^{\beta-1}.
\end{align*}

Next, we claim that the total mass $h\sum_{i=1}^N u_i^k$ is nondecreasing in $k$.
Indeed, by the concavity of $z\mapsto z^{1/\alpha}$ (recall that $\alpha > 1$), 
we have $y^{1/\alpha}-z^{1/\alpha}\ge (1/\alpha)y^{(1-\alpha)/\alpha}(y-z)$ 
for all $y$, $z\ge 0$ and hence,
$$
  \sum_{i=1}^N(u_i^k-u_i^{k-1}) 
	= \sum_{i=1}^N\big((v_i^k)^{1/\alpha}-(v_i^{k-1})^{1/\alpha}\big)
	\ge \frac{1}{\alpha}\sum_{i=1}^N(v_i^k)^{(1-\alpha)/\alpha}(v_i^k-v_i^{k-1}).
$$
Inserting scheme \eqref{1.scheme}, we find that
$$
  \sum_{i=1}^N(u_i^k-u_i^{k-1})
	\ge \frac{\tau}{h^2}\sum_{i=1}^N\big((v_{i+1}^k)^{\beta/\alpha} 
	- 2(v_i^k)^{\beta/\alpha} + (v_{i-1}^k)^{\beta/\alpha}\big) = 0,
$$
since $v_i^k$ satisfies periodic boundary conditions. This shows the claim.

The monotonicity of the total mass and the convergence property 
$h\sum_{i=1}^N u_i^{k_j}\to y^{1/\alpha} = U$ as $k_j\to\infty$ imply that
$h\sum_{i=1}^N u_i^k\to U$ for $k\to\infty$, and the convergence is monotone.
This finishes the proof.
\end{proof}

%%%%%%%%%%%%%%%%%%%%%%%%%%%%%%%%%%%%%%%%%%%%%%%%%%%%%%%%%%%%%%%%%%%%%%%%%%%%%

\section{Numerical examples}\label{sec.num}

We present some numerical results for the porous-medium equation discretized
in the previous section. As initial datum we choose the Barenblatt profile
$$
  u^0(x) = \frac{1}{t_0^{1/(\beta+1)}}\bigg(C-\frac{\beta-1}{2\beta}
	\frac{|x-x_0|^2}{t_0^{2/(\beta+1)}}\bigg)_+^{1/(\beta-1)}, 
$$
where $z_+=\max\{0,z\}$.
We consider two cases. For the slow diffusion case $\beta=4$, we choose
$x_0=0.5$, $t_0=10^{-4}$, and 
$$
  C = \frac{\beta-1}{2\beta}\frac{|x_0|^2}{(t_{\rm end}+t_0)^{2/(\beta+1)}},
	\quad t_{\rm end}=5\cdot 10^{-4}.
$$
The profile will hit the boundary of $\Omega=(0,1)$ at time $t_{\rm end}$.
For the fast diffusion case $\beta=0.5$, we take $x_0=0.5$, $t_0=10^{-2}$,
and $C=t_0^{(\beta-1)/(\beta+1)}$ such that the maximum of the initial
profile equals 1.

Figure \ref{fig.massincrease} illustrates the evolution of the total mass for 
$\alpha=2$, $\beta=0.5$ (left) and $\alpha=3$, $\beta=4$ (right). 
As predicted in Theorem \ref{thm.main}, the total mass is
indeed increasing in time. The mass defect scales well with both the time step $\tau$ 
and the grid size $h$, where the influence of $\tau$ is more prevalent.
 
\begin{figure}[ht]
\includegraphics[width=80mm]{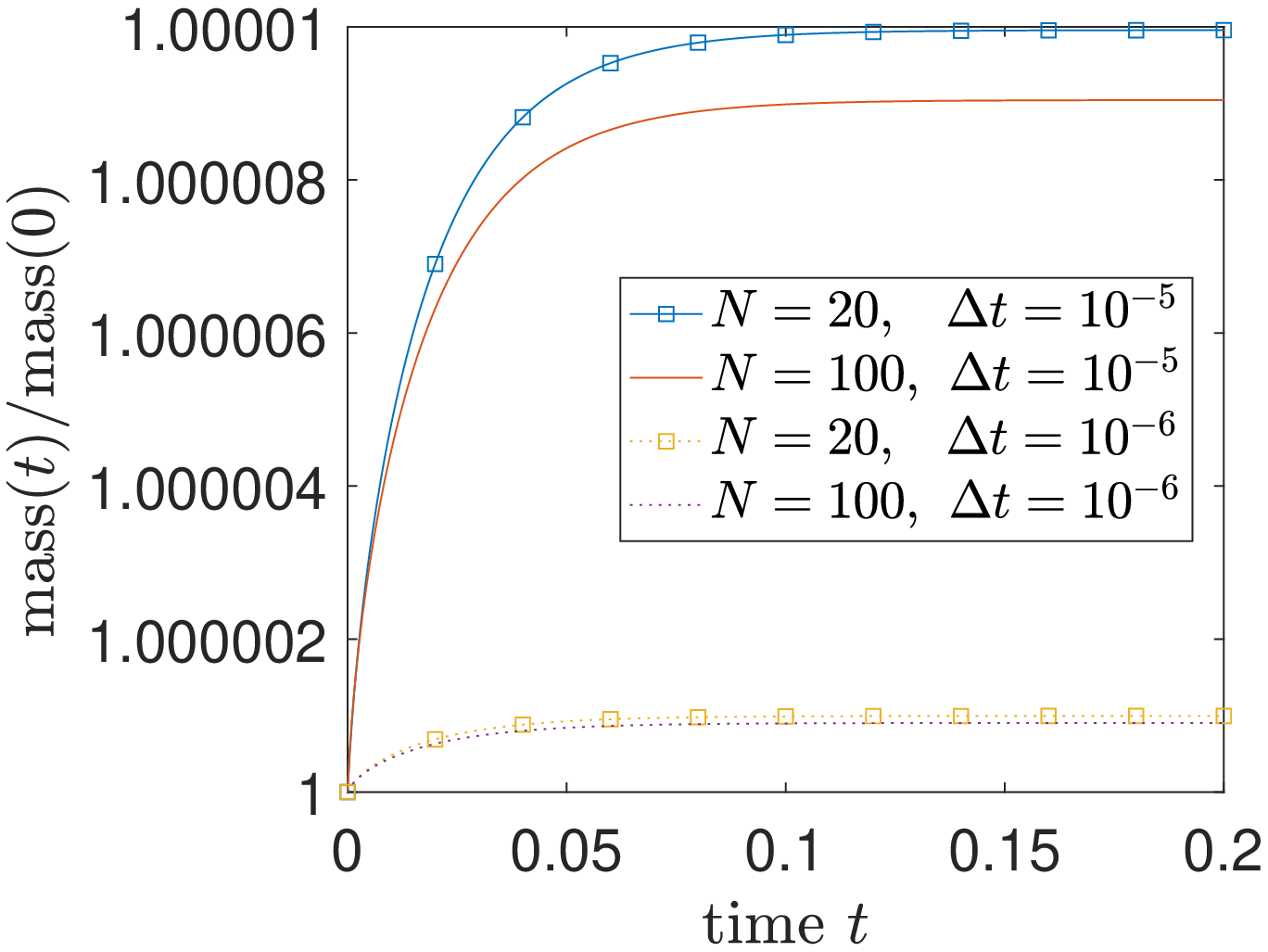}
\includegraphics[width=80mm]{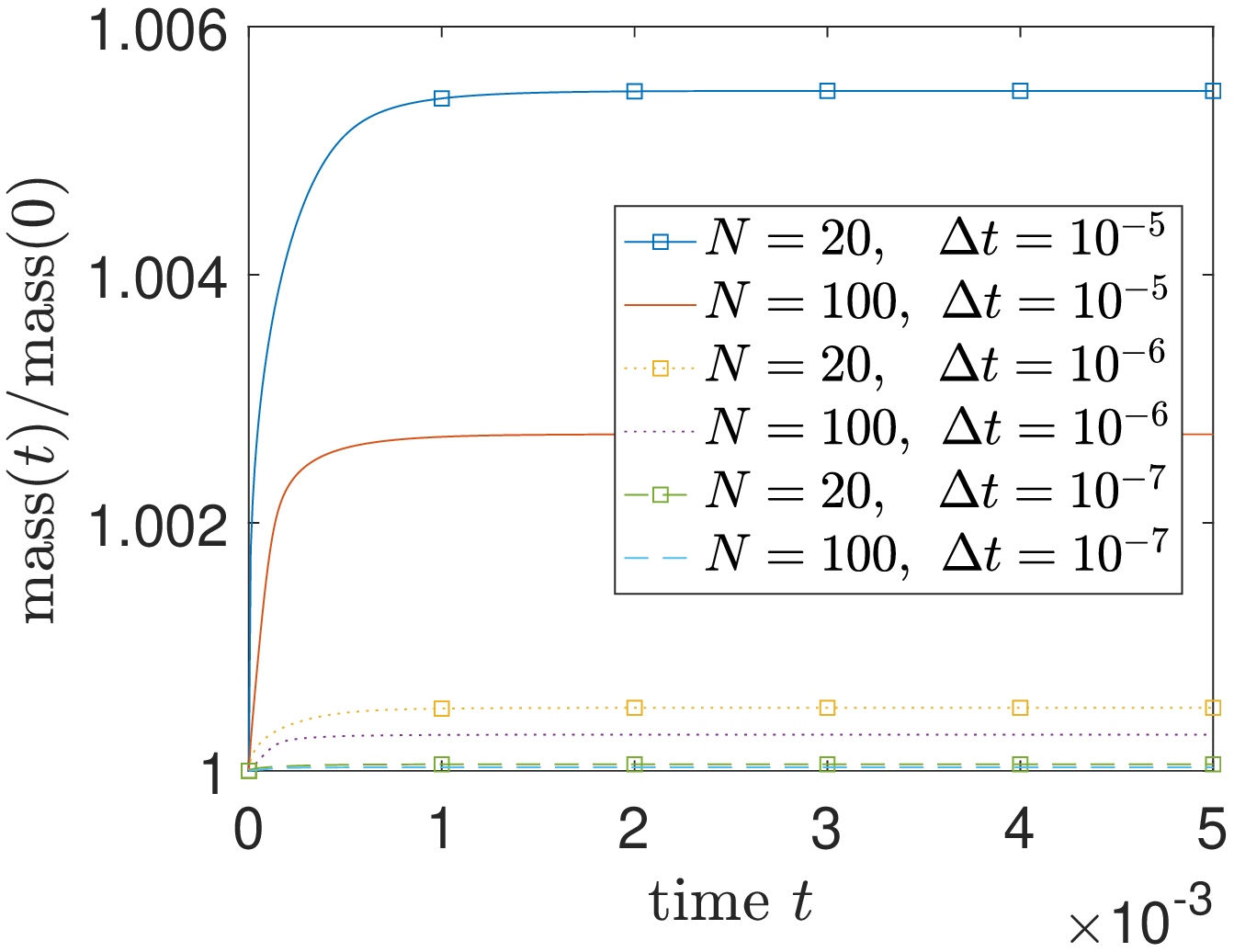}
\caption{Evolution of the total mass for two test scenarios
(left: $\alpha=2$, $\beta=0.5$, right: $\alpha=3$, $\beta=4$).}
\label{fig.massincrease}
\end{figure}

The time decay of the (relative) entropy $\ent$ is shown in Figure 
\ref{fig.relentropy.admissible}
for various space and time steps. We observe that the decay is indeed exponential. 
Here, the steady state $u_\infty$ (which is needed to define the relative entropy) is 
given by $u_\infty=h\sum_{i=0}^N u_i^{k_{\rm max}}$, where $k_{\rm max}$ is the
final time step. This choice clearly depends on the scheme since the mass 
is not conserved.
The relative entropy converges exponentially even when $(\alpha,\beta)$ 
is chosen outside of the admissible
region; see Figure \ref{fig.relentropy.not.admissible}. 

\begin{figure}[ht]
\includegraphics[width=80mm]{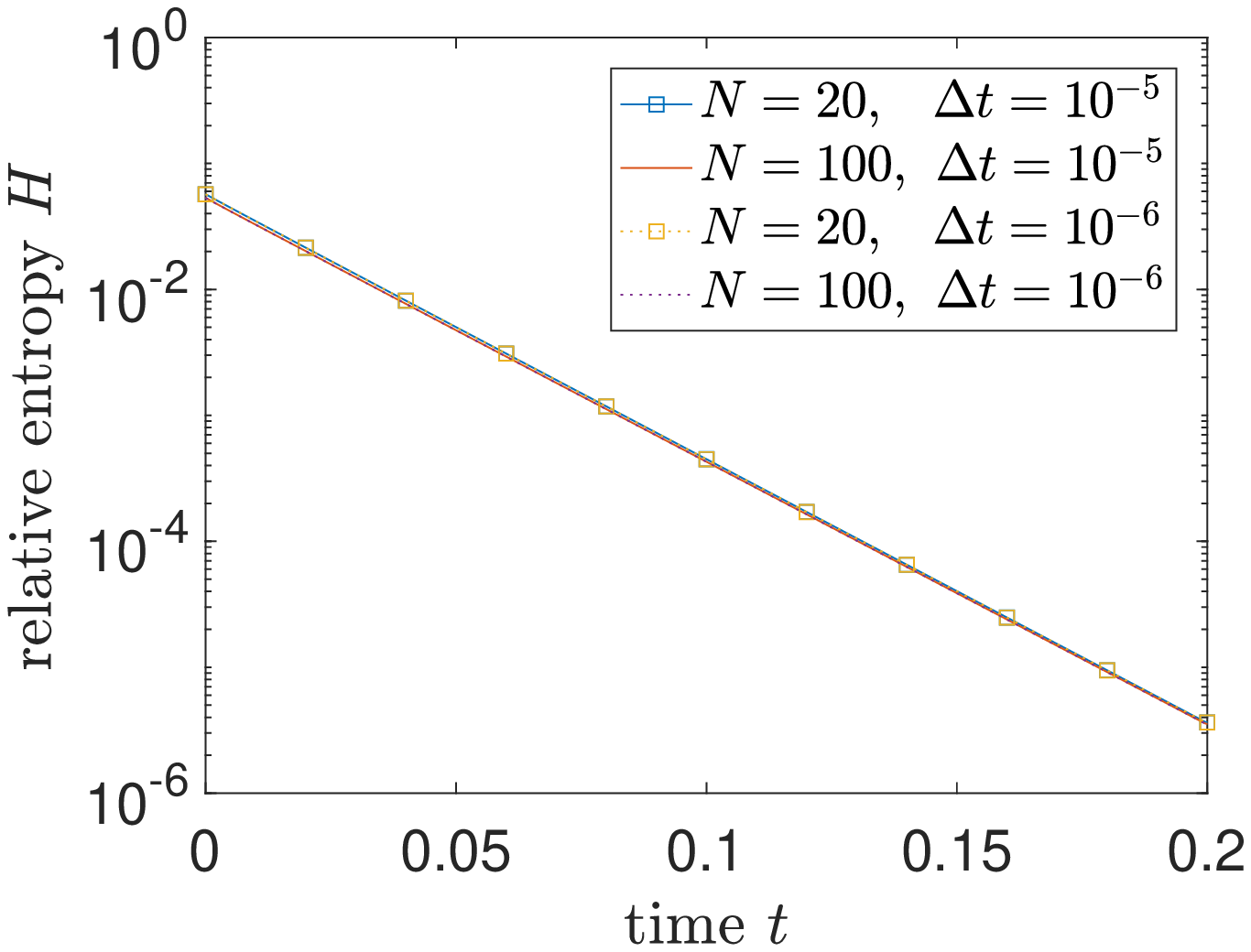}
\includegraphics[width=80mm]{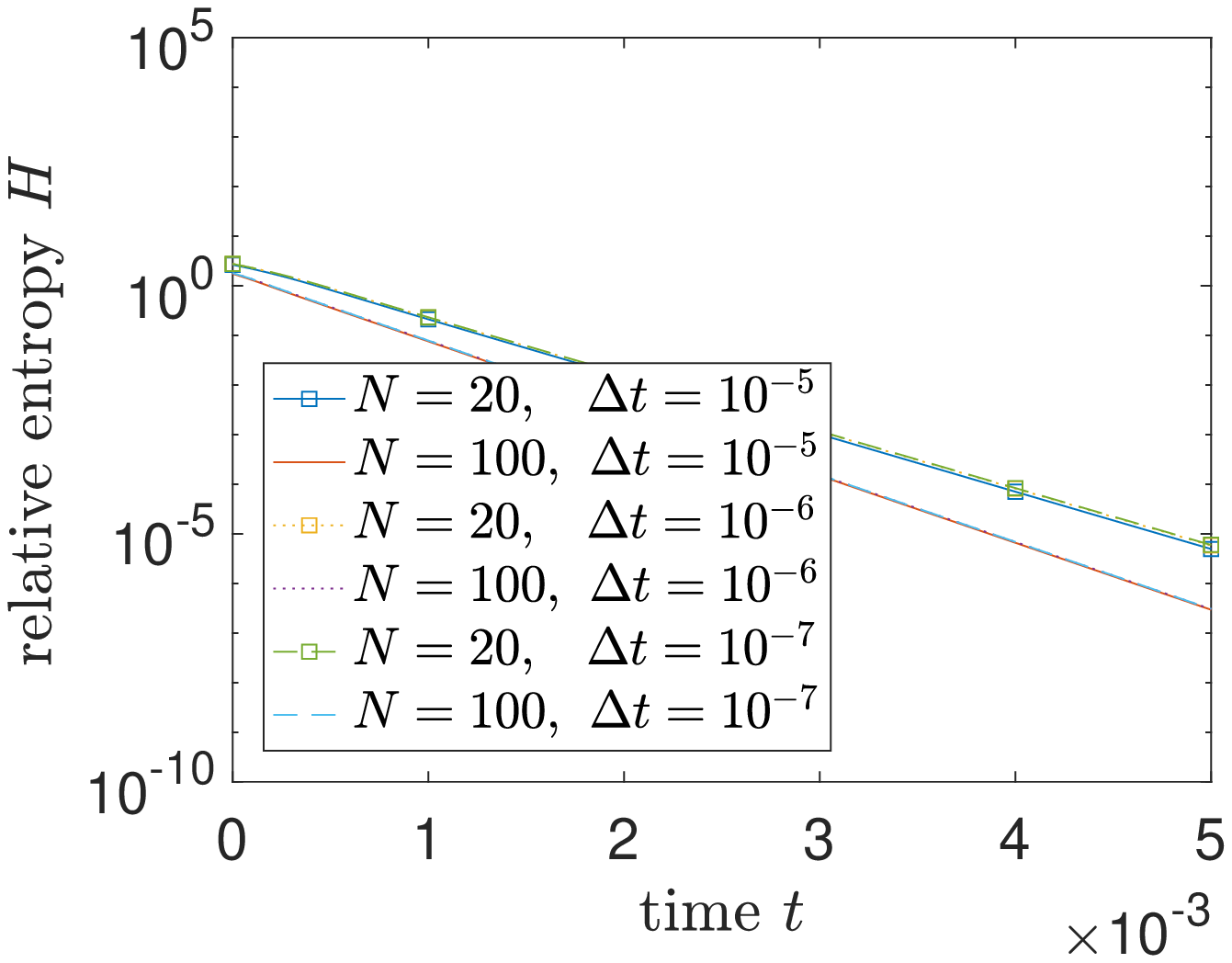}
\caption{Evolution of the relative entropy for two test scenarios
in the admissible region (left: $\alpha=2$, $\beta=0.5$, 
right: $\alpha=3$, $\beta=4$).}
\label{fig.relentropy.admissible}
\end{figure}

\begin{figure}[ht]
\includegraphics[width=90mm]{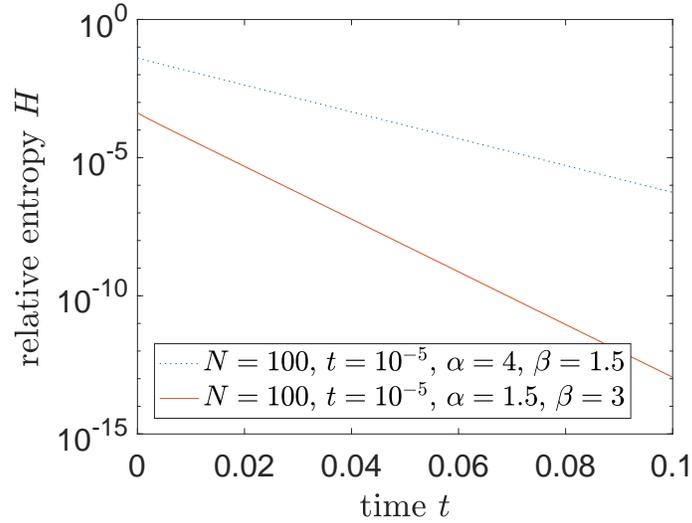}
\caption{Evolution of the relative entropies for $(\alpha,\beta)$ outside of the 
admissible region.}
\label{fig.relentropy.not.admissible}
\end{figure}

%%%%%%%%%%%%%%%%%%%%%%%%%%%%%%%%%%%%%%%%%%%%%%%%%%%%%%%%%%%%%%%%%%%%%%%%%%%%%

\begin{appendix}
\section{Auxiliary inequalities}

\begin{lemma}[Poincar\'e-Wirtinger inequality]\label{lem.poin}
Let $\mathrm{meas}(\T)=1$. 
It holds for all $v\in H^1(\T)$ satisfying $\int_\T udx=0$ that
$$
  \int_\T u^2 dx \le C_P\int_\T u_x^2 dx,
$$
and the constant $C_P=1/(4\pi^2)$ is sharp.
\end{lemma}

\begin{lemma}[Discrete Poincar\'e-Wirtinger inequality]\label{lem.PWI}
Let $N\in\N$, $h=1/N$, $z_0,\ldots,z_{N}\in\R$ satisfying $z_N=z_0$ and 
$\sum_{i=0}^N z_i=0$.
Then
$$
  h\sum_{i=0}^{N-1}z_i^2\le C_p h^{-1}\sum_{i=0}^{N-1}(z_{i+1}-z_i)^2,
$$
where $C_p=h^2/(4\sin^2(h\pi))\ge 1/(4\pi^2)$. This constant is sharp.
\end{lemma}

These lemmas are stated in \cite[Theorem 1]{Shi73}; for proofs see 
\cite[p.~185]{HLP52} (Lemma \ref{lem.poin}) and \cite[Theorem 1]{Sch50}
(Lemma \ref{lem.PWI}).

\begin{lemma}\label{lem.xy}
Let $a$, $b>0$ and $x$, $y\ge 0$. Then
$$
  (x^a-y^a)(x^b-y^b) \le (x^{(a+b)/2}-y^{(a+b)/2})^2
	\le \frac{(a+b)^2}{4ab}(x^a-y^a)(x^b-y^b).
$$
\end{lemma}

\begin{proof}
The second inequality is proven in \cite[Lemma A.3]{CJS16}. For the proof of
the first inequality, we divide it by $y^{a+b}$ and set $z=x/y$. Then
the inequality is equivalent to
$$
  (z^a-1)(z^b-1) \le (z^{(a+b)/2}-1)^2,
$$
which after expansion can be equivalently written as
$(z^{a/2}-z^{b/2})^2\ge 0$, and this is true.
\end{proof}

\end{appendix}

%%%%%%%%%%%%%%%%%%%%%%%%%%%%%%%%%%%%%%%%%%%%%%%%%%%%%%%%%%%%%%%%%%%%%%%%%%%%%


\begin{thebibliography}{11}
\bibitem{AGS05} L.~Ambrosio, N.~Gigli, and G.~Savar\'e. {\em Gradient Flows in Metric 
Spaces and in the Space of Probability Measures}. Lect. Math., Birkh\"auser, 
Basel, 2005.

\bibitem{BaEm85} D.~Bakry and M.~Emery. Diffusions hypercontractives. 
{\em S\'eminaire de probabilit\'es} XIX, 1983/84.
Lecture Notes in Mathmatics, vol. 1123, pp. 177-206. Springer, Berlin, 1985.

\bibitem{BGL14} D.~Bakry, I.~Gentil, and M.~Ledoux. {\em Analysis and Geometry of
Markov Diffusion Operators}. Springer, Cham, 2014.

\bibitem{BCMO16} J.-D.~Benamou, G.~Carlier, Q.~M\'erigot, and E.~Oudet.
Discretization of functionals involving the Monge–Amp\`ere operator.
{\em Numer. Math.} 164 (2016). 611-636.

\bibitem{CDP09} P.~Caputo, P.~Dai Pra, and G.~Posta. Convex entropy decay via
the Bochner-Bakry-Emery approach. {\em Ann. Inst. H. Poincar\'e Prob. Stat.} 
45 (2009), 734-753.

\bibitem{CDGJ06} J.~A.~Carrillo, J.~Dolbeault, I.~Gentil, and A.~J\"ungel.
Entropy-energy inequalities and improved convergence rates for 
nonlinear parabolic equations. {\em Discrete Contin. Dyn. Syst. B}
6 (2006), 1027-1050.

\bibitem{CJSa16} J.~A.~Carrillo, A.~J\"ungel, and M.~C.~Santos. Displacement convexity 
for the entropy in semidiscrete nonlinear Fokker-Planck equations. 
Submitted for publication, 2016. {\tt arXiv:1611.04716}.

\bibitem{CJS16} C.~Chainais-Hillairet, A.~J\"ungel, and S.~Schuchnigg. 
Entropy-dissipative discretization of nonlinear diffusion equations and discrete 
Beckner inequalities. {\em Math. Model. Numer. Anal.} 50 (2016), 135-162.

\bibitem{CJSp16} C.~Chainais-Hillairet, A.~J\"ungel, and P.~Shpartko. A
finite-volume scheme for a spinorial matrix drift-diffusion model for
semiconductors. {\em Numer. Meth. Partial Diff. Eqs.} 32 (2016), 819-846.

\bibitem{CHLZ12} S.~Chow, W.~Huang, Y.~Li, and H.~Zhou. Fokker-Planck equations for a 
free energy functional or Markov process  on a graph. 
{\em Arch. Rational Mech. Anal.} 203 (2012), 969-1008.

\bibitem{Emm09} E.~Emmrich. Variable time-step $\theta$-scheme for nonlinear 
evolution equations governed by a monotone operator.
{\em Calcolo} 46 (2009), 187-210. 

\bibitem{ErMa14} M.~Erbar and J.~Maas. Gradient flow structures for discrete 
porous medium equations. {\em Discrete Contin. Dyn. Sys.} 34 (2014), 1355-1374. 

\bibitem{FaMa16} M.~Fathi and J.~Maas. Entropic Ricci curvature bounds for
discrete interacting systems. {\em Ann. Appl. Prob.} 26 (2016), 1774-1806.

\bibitem{Gli08} A.~Glitzky. Exponential decay of the free energy for discretized
electro-reaction-diffusion systems. {\em Nonlinearity} 21 (2008), 1989-2009.

\bibitem{HLP52} G.~Hardy, J.~Littlewood, and G.~P\'olya. {\em Inequalities}.
Second edition. Cambridge University Press, Cambridge, 1952.

%\bibitem{Jue16} A.~J\"ungel. {\em Entropy Methods for Diffusive Partial Differential
%Equations}. BCAM Springer Briefs, Springer, Cham, 2016.

\bibitem{JuMa06} A.~J\"ungel and D.~Matthes. An algorithmic construction of entropies
in higher-order nonlinear PDEs. {\em Nonlinearity} 19 (2006), 633-659.  

\bibitem{JuMi15} A.~J\"ungel and J.-P.~Mili\v{s}i\'{c}. Entropy dissipative 
one-leg multistep time approximations of nonlinear diffusive equations.
{\em Numer. Meth. Part. Diff. Eqs.} 31 (2015), 1119-1149.

\bibitem{JuSc16} A.~J\"ungel and S.~Schuchnigg. Entropy-dissipating semi-discrete 
Runge-Kutta schemes for nonlinear diffusion equations. To appear in
{\em Commun. Math. Sci.}, 2016. {\tt arXiv:1506.07040.}

\bibitem{JuYu16} A.~J\"ungel and W.~Yue. Discrete Beckner inequalities via the
Bochner-Bakry-Emery approach for Markov chains. To appear in {\em Ann. Appl. Prob.},
2017. 

\bibitem{Maa11} J.~Maas. Gradient flows of the entropy for finite Markov chains. 
{\em J. Funct. Anal.} 261 (2011), 2250-2292.

\bibitem{Mie13} A.~Mielke. Geodesic convexity of the relative entropy in reversible
Markov chains. {\em Calc. Var. Partial Diff. Eqs.} 48 (2013), 1-31.

\bibitem{Sch50} I.~Schoenberg. The finite Fourier series and elementary geometry. 
{\em Amer. Math. Monthly} 57 (1950), 390-404.

\bibitem{Shi73} O.~Shisha. On the discrete version of Wirtinger's inequality.
{\em Amer. Math. Monthly} 80 (1973), 755-760.
\end{thebibliography}
\end{document}